\documentclass{siamart1116}
\usepackage{amsfonts}
\usepackage{amsmath,amssymb}
\usepackage[all]{xy}
\usepackage{tikz}
\usetikzlibrary{arrows.meta}
\usepackage{enumitem}
\usepackage{chemarr}
\usetikzlibrary{trees}
\usetikzlibrary{arrows,shapes,snakes,automata,backgrounds,petri}
\usepackage{accents}
\usepackage{bm}
\usepackage[normalem]{ulem}

\newcommand{\doublehat}[1]{%
\bar{#1}}

\newcommand{\hatt}[1]{%
\hat{#1}}

\newlength\mylen
\usepackage{array} 
\newcolumntype{C}{>{\hfil$}p{\mylen}<{$\hfil}}
\newcolumntype{R}[1]{>{\raggedleft\let\newline\\\arraybackslash\hspace{0pt}}m{#1}}

\mathchardef\mhyphen="2D 

\newtheorem{thm}{Theorem}
\newtheorem{thm1}{Theorem}[section]

\newtheorem{cor}[thm1]{Corollary}

\theorembodyfont{\normalfont} 
\newtheorem{remark}[thm1]{Remark}

\usepackage{pgfplots,comment}
\usepackage{ifthen}
\includecomment{sol}

\newenvironment{myproof}[1] {{\em Proof of {#1}. }}{\hfill$\square$}

\usepackage{tcolorbox}
\tcbuselibrary{theorems}
\definecolor{mycolor}{rgb}{0.122, 0.435, 0.698}
\makeatletter
\newtcbox{\mybox}{nobeforeafter,colframe=mycolor,colback=mycolor!5!white,boxrule=0.5pt,arc=4pt,boxsep=0pt,left=6pt,right=6pt,top=4pt,bottom=4pt,tcbox raise base}
\makeatother

\title{Splitting reactions preserves nondegenerate behaviours in chemical reaction networks}
\author{Murad Banaji\footnotemark[1]}
\begin{document}
\maketitle

\renewcommand{\thefootnote}{\fnsymbol{footnote}}

\footnotetext[1]{Middlesex University London, Department of Design Engineering and Mathematics. {\tt m.banaji@mdx.ac.uk}.}
\renewcommand{\thefootnote}{\arabic{footnote}}

\begin{abstract}
A family of results, referred to as inheritance results, tell us which enlargements of a chemical reaction network (CRN) preserve its capacity for nontrivial behaviours such as multistationarity and oscillation. In this paper, the following inheritance result is proved: under mild assumptions, splitting chemical reactions and inserting complexes involving some new chemical species preserves the capacity of a mass action CRN for multiple nondegenerate equilibria and/or periodic orbits. The claim has been proved previously for equilibria alone; however, the generalisation to include oscillation involves extensive development of rather different techniques. Several inheritance results for multistationarity and oscillation in mass action CRNs, including the main result of this paper, are gathered into a single theorem. Examples are presented showing how these results can be used together to make claims about reaction networks based on knowledge of their subnetworks. The examples include some networks of biological importance.
\end{abstract}

\section{Introduction}

In recent years there has been intense interest in the question of when chemical reaction networks (CRNs) admit nontrivial behaviours such as multistationarity and oscillation. The motivations are often biological \cite{Tysonswitchclock}, and a great variety of mathematical techniques come into play. The work in this area includes a fairly extensive literature on necessary conditions for nontrivial behaviours, and a more modest literature on sufficient conditions for these behaviours. 

Studying oscillation in CRNs often poses greater challenges than studying equilibria. We can rule out oscillation with the help of Lyapunov functions, or using Bendixson-Dulac type criteria, or via other conditions which guarantee convergence of all orbits to equilibria. But if we wish to claim that oscillation {\em does} occur in a given CRN, the techniques available are more limited. Apart from numerical simulation, we find the following approaches in recent papers: (i) proving the occurrence of Andronov--Hopf bifurcations and other bifurcations leading to oscillations (\cite{errami2015,hellrendall2016,ConradiERK2,boroshofbauerdef1,kreusserrendall2021,banajiborosHopf} for example); (ii) use of Poincar\'e-Bendixson type theorems (e.g., \cite{HalSmithJMC,boroshofbauerplanar}); (iii) tests showing the occurrence of relaxation oscillations \cite{gedeonsontag2007}; and (iv) inheritance results as in this paper.

Theorems on inheritance in CRNs, such as those in \cite{Conradi19175,joshishiu,feliuwiufInterface2013,banajipanteaMPNE,banajiCRNosci,banajiCRNosci1,banajiboroshofbauer,GrossJMB}, provide some answers to the following question: under what circumstances can we be sure that a CRN will display some interesting dynamical behaviour, simply based on the presence of a smaller network embedded within it? The theorems also suggest potentially useful ways of decomposing complex reaction networks into simpler subnetworks. 

Inheritance results are often naturally coupled with other approaches, for example involving bifurcations. In \cite{hellrendall2016} and \cite{kreusserrendall2021} bifurcations giving rise to oscillation are proved to be inherited by relatively large networks from simpler subnetworks. This approach is also discussed and demonstrated by example in \cite{banajiboroshofbauer} and \cite{banajiborosHopf}.

One feature of inheritance results which adds to their power is their ``scalability''. Whereas directly proving the occurrence of interesting dynamical behaviours becomes increasingly difficult as network size increases, demonstrating that interesing behaviours must occur as a consequence of the presence of some subnetwork can remain tractable. For example, oscillation was found by Qiao {\em et al} \cite{Qiao.2007aa} in numerical simulations of a fairly large model of the biologically important MAPK signalling pathway; but it was later proved by Hell and Rendall \cite{hellrendall2016} that oscillations must occur in this network by inheritance from its subnetworks. In \cite{ConradiERK2}, Conradi {\em et al} observe that the size and complexity of another biologically important network, the ERK network, makes direct confirmation that it displays a Andronov--Hopf bifurcation computationally challenging. Using the theory in this paper, however, we are able to confirm in Section~\ref{secERK} that the network in question must admit oscillation as a consequence of how it is built from a smaller oscillatory network.

Another motivation for this work stems from various common practices in biological modelling. As observed in the review paper \cite{Rosenfeld2011}, a single biochemical process may often be treated as either a single reaction or a complicated system of reactions. Modellers making the choice to simplify processes, for example by omitting chemical species or reactions, are aware that they may lose interesting dynamical behaviours by doing so. However, they may hope that simplifying a network cannot {\em introduce} interesting new behaviours into models. This assumption is risky. Various counterexamples demonstrate that a simplified CRN, for example with some reactions removed, can sometimes display nontrivial behaviours which did not occur in the larger network (such examples appear in \cite{banajipanteaMPNE,banajiCRNosci,GrossJMB}). Theorems on inheritance in CRNs provide a list of allowed simplifications which are incapable of introducing complex behaviours into models of a CRN, provided various assumptions on the kinetics in the original and simplified networks are satisfied. 

We now turn to the specific motivations for the main result of this paper, namely Theorem~\ref{mainthm} below, which concerns building CRNs by adding ``intermediate'' species or complexes into reactions. In \cite{feliuwiufInterface2013}, Feliu and Wiuf observe that it is common amongst biological modellers to ignore intermediate species, especially if they are not readily measurable. Motivated by this observation, in Theorem~6 of \cite{banajipanteaMPNE} it was proved that inserting intermediate complexes involving new species into reactions preserves the capacity of a mass action CRN for multiple nondegenerate equilibria, provided the new species enter nontrivially into the enlarged system. In Theorem~\ref{mainthm} of this paper this result is extended to oscillation. Informally, we find that:
\begin{quote}
If a mass action CRN admits nondegenerate oscillation, and we replace some reactions with chains of reactions involving new species, then under mild conditions the resulting enlarged CRN admits nondegenerate oscillation. 
\end{quote}
Whereas the related result for equilibria could be proved largely by calling on the implicit function theorem, more global singular perturbation theory approaches are required to deal with periodic orbits. The usefulness of geometric singular perturbation theory (GSPT) for proving the occurrence of bifurcations and oscillations in chemical reaction networks has previously been demonstrated in both specific applications and general theorems \cite{hellrendall2016,banajiCRNosci,banajiCRNosci1,kreusserrendall2021}.

The practical utility of Theorem~\ref{mainthm}, especially when used in conjunction with previous inheritance results, is demonstrated in a corollary and via examples. In Corollary~\ref{corenz} we show how introducing ``enzymatic mechanisms'' into mass action CRNs preserves their capacity for nondegenerate multistationarity and oscillation. In Section~\ref{secERK} we are able to answer questions posed by Conradi, Obatake, and co-authors in \cite{obatake}~and~\cite{ConradiERK2} about oscillation in the biologically important ERK network. And in Section~\ref{secMAPK} we show that another biologically important network, the MAPK network with negative feedback, inherits oscillation from some of its subnetworks. The proof of Theorem~\ref{mainthm} also completes the process of proving analogues of all the results in \cite{banajipanteaMPNE} with ``multistationarity'' replaced with ``oscillation''. With the goal of providing a convenient reference to these results, in Theorem~\ref{thminherit} we summarise and slightly generalise available inheritance results for both multistationarity and oscillation in mass action systems.

\section{Statement of the main result}

Relevant background and terminology needed in order to state the main result are outlined only briefly; the reader is referred to \cite{banajiCRNosci,banajiCRNosci1} for more detail. 

Given a list of chemical species $\mathrm{X}= \mathrm{X}_1, \ldots, \mathrm{X}_n$ and a real vector $(c_1,\ldots, c_n)^\mathrm{t}$, we write $c \cdot \mathrm{X}$ for the formal sum $c_1\mathrm{X}_1+\cdots +c_n\mathrm{X}_n$, termed a {\em complex}. An irreversible chemical reaction involves conversion of one complex, termed the {\em reactant complex}, into another, termed the {\em product complex}. A chemical reaction network is a set of chemical reactions on some set of species. It is helpful to assume that both species and reactions are given some arbitrary, but fixed, ordering, and that all chemical reactions are irreversible. 

The concentration of species $\mathrm{X}_i$ is denoted by $x_i$, and the concentration vector $x:=(x_1, \ldots, x_n)^\mathrm{t}$ is assumed to vary in the {\em nonnegative orthant}, namely, $\mathbb{R}^n_{\geq 0}:=\{y \in \mathbb{R}^n\,\colon\, y_i \geq 0\,\,(i=1, \ldots, n)\}$. The interior of the nonnegative orthant, denoted by $\mathbb{R}^n_{+}$, is referred to as the {\em positive orthant}, and any subset of $\mathbb{R}^n_{+}$ is referred to as {\em positive}. 

A CRN involving $r$ reactions on species $\mathrm{X} = \mathrm{X}_1, \ldots, \mathrm{X}_n$ gives rise to an ordinary differential equation (ODE) model of the form $\dot x = \Gamma v(x)$ where: (i) $\Gamma$ is the {\em stoichiometric matrix} of the CRN with dimension $n \times r$, whose $ij$th entry tells us the net production of species $i$ in reaction $j$; and (ii) $v$ is the {\em rate function} of the reaction, always assumed to be defined (at least) on $\mathbb{R}^n_{+}$, and taking values in $\mathbb{R}^{r}$. Common choices of {\em kinetics}, namely families of functions to which $v$ may belong, are discussed in \cite{banajiCRNosci}. In a CRN with mass action kinetics, the rate of a reaction with reactant complex $c\cdot \mathrm{X}$ takes the form $kx_1^{c_1}\cdots x_n^{c_n}$, abbreviated as $k x^c$. The constant $k$ is the {\em rate constant} for the reaction. A mass action CRN is said to ``admit'' some behaviour if this behaviour occurs for some choice of rate constants.

The image of the stoichiometric matrix $\Gamma$ is termed the {\em stoichiometric subspace} of the CRN, and the dimension of $\mathrm{im}\,\Gamma$ is referred to as the {\em rank} of the CRN. The nonempty intersection of any coset of $\mathrm{im}\,\Gamma$ with $\mathbb{R}^n_{\geq 0}$ is termed a {\em stoichiometric class} of the CRN (also termed a {\em stoichiometry class} or {\em stoichiometric compatibility class}). Here we are interested in the positive parts of stoichiometric classes, termed {\em positive stoichiometric classes}. These are locally invariant, and each positive equilibrium or periodic orbit of a CRN lies in some positive stoichiometric class. 

We refer to an equilibrium or periodic orbit of a CRN as being {\em nondegenerate} (resp., {\em hyperbolic}, resp., {\em linearly stable}) if it is nondegenerate, (resp., hyperbolic, resp., linearly stable) relative to its stoichiometric class. More precisely, consider a CRN with stoichiometric matrix $\Gamma$ and some $C^1$ rate function $v$, giving the system of ODEs $\dot x = \Gamma v(x)$. An equilibrium $p$ of this system is nondegenerate (resp., hyperbolic, resp., linearly stable) if all eigenvalues of $\Gamma Dv(p)$ corresponding to generalised eigenspaces spanning $\mathrm{im}\,\Gamma$ are nonzero (resp., avoid the imaginary axis, resp., have negative real parts). A periodic orbit $\mathcal{O}$ is nondegenerate (resp., hyperbolic, resp., linearly stable) if all of its Floquet multipliers relative to $\mathrm{im}\,\Gamma$, except for the trivial multiplier which always equals $1$, are distinct from $1$ (resp., avoid the unit circle, resp., lie inside the unit circle). Note a change in terminology from \cite{banajiCRNosci,banajiCRNosci1} where a periodic orbit was referred to as ``nondegenerate'' only if it was hyperbolic relative to its stoichiometric class. 

We are now ready to state the main result. 

\begin{thm}
\label{mainthm}
Consider a CRN $\mathcal{R}$ on species $\mathrm{X} = \mathrm{X}_1, \ldots, \mathrm{X}_n$ with mass action kinetics. Let $m \geq 1$, and let $a_i \cdot \mathrm{X} \rightarrow b_i \cdot \mathrm{X}\,\, (i=1,\ldots,m)$ be any reactions of $\mathcal{R}$. Let $\mathcal{R}'$ be a new CRN created from $\mathcal{R}$ by replacing each of the reactions $a_i \cdot \mathrm{X} \rightarrow b_i \cdot \mathrm{X}$ with a chain
\[
a_i \cdot \mathrm{X} \rightarrow c_i \cdot \mathrm{X} + \beta_i\cdot \mathrm{Y} \rightarrow b_i \cdot \mathrm{X},\,\,(i=1,\ldots,m)\,.
\]
Here, $\mathrm{Y}$ is a list of $m+k$ new species ($k \geq 0$), $c_i$ and $\beta_i$ are arbitrary nonnegative vectors of lengths $n$ and $m+k$ respectively, and we assume that the new species $\mathrm{Y}$ enter nontrivially into $\mathcal{R}'$ in the sense that $\beta : = (\beta_1|\beta_2|\cdots|\beta_m)$ has rank $m$.  

Now suppose that $\mathcal{R}$ admits, on some stoichiometric class, $0 \leq r_1 < \infty$ positive nondegenerate equilibria $\mathcal{O}_1, \ldots, \mathcal{O}_{r_1}$, and $0 \leq r_2 < \infty$ positive nondegenerate periodic orbits $\mathcal{O}_{r_1+1}, \ldots, \mathcal{O}_{r_1+r_2}$. Then, with mass action kinetics, $\mathcal{R}'$ admits, on some stoichiometric class, at least $r_1$ positive nondegenerate equilibria, say, $\mathcal{O}'_1, \ldots, \mathcal{O}'_{r_1}$, and at least $r_2$ positive nondegenerate periodic orbits, say, $\mathcal{O}'_{r_1+1}, \ldots, \mathcal{O}'_{r_1+r_2}$. Rate constants for $\mathcal{R}'$ may be chosen to ensure that, whenever $\mathcal{O}_i$ was hyperbolic (resp., linearly stable), then the same holds for $\mathcal{O}'_i$.
\end{thm}

Theorem~\ref{mainthm} is proved in Appendix~\ref{appproof} where an extended example is also used to illustrate the steps in the proof. Several constructions in the proof follow that of Theorem~1 in \cite{banajiCRNosci1}, although there are some important technical differences. 

\begin{remark}[The inheritance of oscillation]
\label{remnovel}
What distinguishes Theorem~\ref{mainthm} here from the related theorem for multistationarity in \cite{banajipanteaMPNE} is the {\em global} nature of the result. In order to prove that periodic orbits ``survive'' when intermediate complexes are added into a CRN, we need to be able to control the dynamics of the enlarged network not just at isolated points, but over arbitrary compact subsets of some stoichiometric class. 
\end{remark}

\begin{remark}[Going beyond mass action kinetics]
\label{remkin}
Whereas, for simplicity, Theorem~\ref{mainthm} is stated for networks with mass action kinetics, its proof implies an immediate extension. In the proof, the reactions with new reactant complexes, inserted in the splitting process, are required to have mass action kinetics; however, the original rates of reaction of the CRN prior to enlargement can be drawn from any class of sufficiently differentiable functions which are positive on the positive orthant. It follows that the conclusions of the theorem hold if we replace ``mass action'' with any sufficiently differentiable class of rate functions which include mass action rate functions as a special case -- for example, positive general kinetics (see \cite{banajiCRNosci} for a definition).
\end{remark}

\begin{remark}[Applications of Theorem~\ref{mainthm} which are not immediately obvious]
\label{remcons}
Following Remark~4.4 in \cite{banajipanteaMPNE}, the scope of the theorem is broader than it at first appears. For example, we may consider a single reaction $a_i \cdot \mathrm{X} \rightarrow b_i\cdot \mathrm{X}$ as a set of $m$ such reactions, each with rate $\frac{1}{m}$ times the original rate. This does not affect the associated differential equations and can be done while remaining in any reasonable class of kinetics (and in particular mass action kinetics). With this preliminary step, a single reaction $a_i \cdot \mathrm{X} \rightarrow b_i\cdot \mathrm{X}$ may be split multiple times and acquire multiple intermediate complexes. Another construction we may employ is to first add a trivial reaction $a_i \cdot \mathrm{X} \rightarrow a_i\cdot \mathrm{X}$ to $\mathcal{R}$ which has no effect on the dynamics, and then ``split'' this trivial reaction using Theorem~\ref{mainthm}; the net effect is to add the reversible reaction $a_i \cdot \mathrm{X} \rightleftharpoons c_i\cdot \mathrm{X} + \beta_i\cdot \mathrm{Y}$ to $\mathcal{R}$. Thus some instances of Theorem~5 in \cite{banajipanteaMPNE} and Theorem~1 in \cite{banajiCRNosci1} follow as consequences of Theorem~\ref{mainthm} here. 
\end{remark}

\begin{remark}[The condition that the new species figure nontrivially in the enlarged CRN]
\label{remnewspec}
The condition in Theorem~\ref{mainthm} that the matrix $\beta$ has rank $m$ also appears in Theorem~1 of \cite{banajiCRNosci1} and Theorems~5~and~6~of~\cite{banajipanteaMPNE}. In all cases, it is equivalent to the requirement that the submatrix of the new stoichiometric matrix corresponding to the added species must have rank $m$, where $m$ is the number of reactions which are split (here and in Theorem~6~of~\cite{banajipanteaMPNE}), or the number of reversible reactions added (in \cite{banajiCRNosci1} and in Theorem~5~of~\cite{banajipanteaMPNE}). Although it implies that the rank of the CRN as a whole increases by $m$, it is {\em not} equivalent to this condition. 
\end{remark}

\section{A summary of some inheritance results in a single theorem}

From here on, given two mass action CRNs $\mathcal{R}$ and $\mathcal{R}'$, we use the phrase 
\begin{quote}
``$\mathcal{R}'$ inherits nondegenerate equilibria and periodic orbits from $\mathcal{R}$''
\end{quote}
to signify the conclusion of Theorem~\ref{mainthm}. It means that if $\mathcal{R}$ admits, on some stoichiometric class, $r_1$ positive nondegenerate equilibria and $r_2$ positive nondegenerate periodic orbits, then the same holds for $\mathcal{R}'$. Moreover the $r_1+r_2$ nondegenerate equilibria and periodic orbits constructed on some stoichiometric class of $\mathcal{R}'$ are in natural one-to-one correspondence with those of $\mathcal{R}$, and hyperbolicity/linear stability of any one of these limit sets for $\mathcal{R}$ (for some choice of rate constants) implies the same for the corresponding limit set of $\mathcal{R}'$ (for some choice of rate constants).

Consider a CRN $\mathcal{R}$ and the following six enlargements:
\begin{enumerate}
\item[E1.] {\em A new linearly dependent reaction.} We add to $\mathcal{R}$ a new reaction involving only existing chemical species of $\mathcal{R}$, and in such a way that the rank of $\mathcal{R}$ remains unchanged.
\item[E2.] {\em The fully open extension.} We add in (if absent) all chemical reactions of the form $0 \rightarrow \mathrm{X}_i$ and $\mathrm{X}_i \rightarrow 0$ for each chemical species $\mathrm{X}_i$ of $\mathcal{R}$.
\item[E3.] {\em A new linearly dependent species.} We add a new chemical species into the reactions of $\mathcal{R}$, in such a way that the rank of $\mathcal{R}$ remains unchanged.
\item[E4.] {\em A new species and its inflow-outflow.} We add a new chemical species, say $\mathrm{Y}$, into some or all reactions of $\mathcal{R}$, and also add the inflow and outflow reactions $0 \rightarrow \mathrm{Y}$ and $\mathrm{Y} \rightarrow 0$. 
\item[E5.] {\em New reversible reactions involving new species.} We add $m \geq 1$ new reversible reactions into $\mathcal{R}$ involving $m+k$ new chemical species ($k \geq 0$). Moreover, the new species figure nontrivially in the enlarged CRN in the sense of Remark~\ref{remnewspec}.
\item[E6.] {\em Splitting reactions.} We split $m \geq 1$ chemical reactions of $\mathcal{R}$ and insert complexes involving $m+k$ new species ($k \geq 0$). Moreover, the new species figure nontrivially in the enlarged CRN in the sense of Remark~\ref{remnewspec}.
\end{enumerate}

The following summary of results is presented under the assumption of mass action kinetics, although many of the individual results can be proved under a variety of other kinetic assumptions. 

\begin{thm}
\label{thminherit}
Let $\mathcal{R}$ be a CRN with mass action kinetics, and let $\mathcal{R}'$ be a CRN obtained from $\mathcal{R}$ by any finite sequence, possibly empty, of enlargements E1--E6 above. Then $\mathcal{R}'$ inherits nondegenerate equilibria and periodic orbits from $\mathcal{R}$.
\end{thm}
\begin{proof}
The result follows from the proofs of several theorems including Theorem~\ref{mainthm} above. In particular:
\begin{itemize}
\item Theorem~1 in \cite{banajiboroshofbauer} treats enlargement E3 for equilibria and periodic orbits.
\item Theorems~1, 2, 4, 5, and 6 in \cite{banajipanteaMPNE} treat enlargements E1, E2, E4, E5 and E6 for equilbiria.
\item Theorems~1, 2 and 4 in \cite{banajiCRNosci} treat enlargements E1, E2, and E4 for periodic orbits.
\item Theorem~1 in \cite{banajiCRNosci1} treats enlargement E5 for periodic orbits.
\item Theorem~1 here treats enlargement E6 for periodic orbits (and equilibria).
\end{itemize}
Several of the previous results were stated in more restricted terms, and combining them into the present claim requires a couple of easy observations. 

\begin{enumerate}[align=left,leftmargin=*]
\item Results were stated for nondegenerate and linearly stable equilibria in \cite{banajipanteaMPNE}, and for periodic orbits which were hyperbolic or linearly stable in \cite{banajiCRNosci,banajiCRNosci1}. It is easy to see that all the proofs of existence of limit sets in the enlarged CRNs were continuation results relying only on nondegeneracy of the limit sets in the sense used here. Further, provided we choose the continuation parameter sufficiently small, then the proofs demonstrated that hyperbolic (resp., linearly stable) equilibria and periodic orbits remain hyperbolic (resp., linearly stable) in the enlarged CRN. This claim is easy for enlargements E1 and E3 which preserve the rank of the CRN, and some direct calculation suffices in the case of enlargement E2. On the other hand, enlargements E4, E5 and E6 give rise to singular perturbation problems \cite{JonesSingular} and the lifted equilibria (resp., periodic orbits) of the enlarged systems have additional eigenvalues (resp., Floquet multipliers) corresponding to directions transverse to a slow manifold. However, as in the proof of Theorem~\ref{mainthm} in Appendix~\ref{appproof}, the additional nontrivial eigenvalues (resp., Floquet multipliers) are always in the left half plane (resp., inside the unit circle) in $\mathbb{C}$. 

\item Results in \cite{banajipanteaMPNE} were stated for two equilibria; and results in \cite{banajiCRNosci,banajiCRNosci1} were stated for a single periodic orbit. In order to combine these into a single result involving an arbitrary finite number of equilibria and/or periodic orbits on some stoichiometric class, we can follow the approach in the proof of Theorem~1 in \cite{banajiboroshofbauer} and Theorem~\ref{mainthm} here. In brief, we surround the nondegenerate equilibria and periodic orbits by disjoint, compact, positive neighbourhoods on their stoichiometric class, apply the techniques of proof within each of these neighbourhoods, and take as an upper bound on our continuation parameter the minimum of the bounds required for continuation (and, if required, hyperbolicity or linear stability) of each limit set. 
\end{enumerate}
\end{proof}

\subsection{Enzymatic mechanisms}

The following result is an easy corollary of Theorem~\ref{thminherit}, and extends Corollary~4.7 in \cite{banajipanteaMPNE}. It tells us that replacing a set of reactions in a mass action CRN with enzyme mediated mechanisms preserves the CRN's capacity for nondegenerate multistationarity or oscillation. In fact, a single enzyme may mediate more than one process. 

\begin{cor}[Adding enzymatic mechanisms]
\label{corenz}
Let $\mathcal{R}$ be a CRN with mass action kinetics on species $\mathrm{X} = \mathrm{X}_1, \ldots, \mathrm{X}_n$. Let $a_i \cdot \mathrm{X} \rightarrow b_i \cdot \mathrm{X}\,\, (i=1,\ldots,m)$ be any $m$ reactions in $\mathcal{R}$. Let $\mathrm{E}, \mathrm{I}_1, \ldots, \mathrm{I}_m$ be new species, not occurring in $\cal R$, and let $c_i \ge 0$ ($i = 1, \ldots, m$). Suppose we create ${\cal R}'$ from ${\cal R}$ by replacing each of the reactions
\[
a_i\cdot \mathrm{X}\to b_i\cdot \mathrm{X} \quad \mbox{in $\mathcal{R}$ with a chain} \quad c_i\mathrm{E}+ a_i\cdot \mathrm{X}\rightleftharpoons \mathrm{I}_i\to  c_i\mathrm{E}+b_i\cdot \mathrm{X},\,\,(i=1,\ldots,m)\,.
\]
Then $\mathcal{R}'$ inherits nondegenerate equilibria and periodic orbits from $\mathcal{R}$.
\end{cor} 

\begin{proof}
We successively enlarge $\mathcal{R}$, eventually obtaining $\mathcal{R}'$, and using only enlargements E3, E6 and E1 above as follows:
\begin{enumerate}[align=left,leftmargin=*]
\item Let $\mathcal{R}^*$ be the CRN created from $\mathcal{R}$ by adding the species $\mathrm{E}$ into both sides of the relevant reactions of $\mathcal{R}$, namely:
\[
c_i\mathrm{E}+ a_i\cdot \mathrm{X}\to  c_i\mathrm{E}+b_i\cdot \mathrm{X}\,\, \text{ replaces }\,\, a_i\cdot \mathrm{X}\to b_i\cdot \mathrm{X},\quad (i=1,\ldots,m).
\]
Note that this process leaves the rank of $\mathcal{R}$ unchanged, and so is an instance of enlargement E3. 
\item Let $\mathcal{R}^{**}$ be the CRN created from $\mathcal{R}^*$ by adding the intermediate species $\mathrm{I}_i$, namely
\[
c_i\mathrm{E}+ a_i\cdot \mathrm{X}\to \mathrm{I}_i\to  c_i\mathrm{E}+b_i\cdot \mathrm{X}\,\, \text{ replaces } \,\, c_i\mathrm{E}+ a_i\cdot \mathrm{X}\to  c_i\mathrm{E}+b_i\cdot \mathrm{X},\quad (i=1,\ldots,m).
\]
This corresponds to enlargement E6; the nondegeneracy condition is easily seen to be satisfied. 
\item Finally, $\mathcal{R}'$ is obtained from $\mathcal{R}^{**}$ by adding the reverse reactions $\mathrm{I}_i\to c_i\mathrm{E}+a_i\cdot \mathrm{X}$, namely:
\[
c_i\mathrm{E}+ a_i\cdot \mathrm{X}\rightleftharpoons \mathrm{I}_i \,\, \text{ replaces }\,\, c_i\mathrm{E}+ a_i\cdot \mathrm{X}\to \mathrm{I}_i,\quad (i=1,\ldots,m)\,.
\]
This involves $m$ applications of enlargement E1, as adding the reverse of a reaction does not affect the rank of a CRN.
\end{enumerate}
The result is now immediate from Theorem~\ref{thminherit}. The proof is illustrated in Figure~\ref{figenz}.
\end{proof}

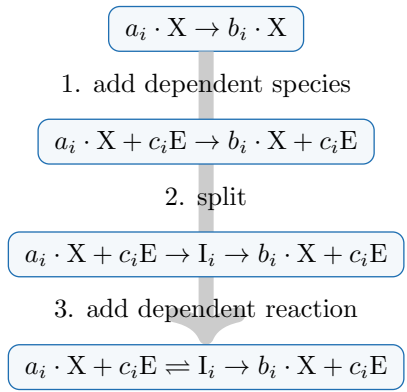
\begin{figure}[h]
\begin{center}
\begin{tikzpicture}[domain=0:4,scale=0.75]

\node at (0,6) {\mybox{$a_i\cdot \mathrm{X} \rightarrow b_i\cdot \mathrm{X}$}};
\draw[->,black!20,line width=0.2cm] (0,5.6) -- (0,0.4);

\node at (0,5) {1. add dependent species};
\node at (0,4) {\mybox{$a_i\cdot \mathrm{X} + c_i\mathrm{E} \rightarrow b_i\cdot \mathrm{X} + c_i\mathrm{E}$}};
\node at (0,3) {2. split};
\node at (0,2) {\mybox{$a_i\cdot \mathrm{X} + c_i\mathrm{E} \rightarrow \mathrm{I}_i \rightarrow b_i\cdot \mathrm{X} + c_i\mathrm{E}$}};
\node at (0,1) {3. add dependent reaction};

\node at (0,0) {\mybox{$a_i\cdot \mathrm{X} + c_i\mathrm{E} \rightleftharpoons \mathrm{I}_i \rightarrow b_i\cdot \mathrm{X} + c_i\mathrm{E}$}};

\end{tikzpicture}
\end{center}
\caption{\label{figenz} A schematic of the proof of Corollary~\ref{corenz}. The process of adding enzymatic mechanisms to a reaction network breaks down into three steps: adding a dependent species, adding intermediates into reactions, and adding dependent reactions.}
\end{figure}

\begin{remark}
Corollary~\ref{corenz} has the following interpretation in terms of modelling CRNs arising in biochemistry. It tells us that if a mass action CRN does {\em not} admit nondegenerate oscillation or multistationarity, and we caricature the CRN by collapsing enzyme-mediated processes into single reactions, then we cannot introduce nondegenerate multistationarity or oscillation by this simplification, provided we use mass action kinetics for the collapsed reactions. Note that if an enzyme is presumed to mediate more than one reaction, then we must simultaneously collapse {\em all} mechanisms in which this enzyme figures for this claim to hold. Note also that Corollary~\ref{corenz} does not guarantee that replacing an enzymatic mechanism with a Michaelis-Menten scheme cannot introduce nondegenerate multistationarity or oscillation.
\end{remark}

\section{Examples}

We present four examples applying the results above. While the Examples in Sections~\ref{secbasicex}~and~\ref{secmultiinherit} involve small networks, and are designed to illustrate the key ideas behind Theorem~\ref{mainthm} and its proof, the examples in Sections~\ref{secERK}~and~\ref{secMAPK} focus on systems of biological importance, and use combinations of the results gathered in Theorem~\ref{thminherit}. The latter two examples involve fairly large CRNs, and so demonstrate the natural ``scalability'' of inheritance results, mentioned in the introduction. The main focus is on the inheritance of oscillation; however, the example in Section~\ref{secmultiinherit} uses more of the power of Theorem~\ref{mainthm}, demonstrating the inheritance of multiple, coexisting, nondegenerate, omega limit sets, including both periodic orbits and equilibria.

\subsection{A basic example illustrating Theorem~\ref{mainthm}}
\label{secbasicex}

The following CRN appeared in \cite{banajiCRNosci} as an example of a CRN admitting stable oscillation with mass action kinetics. 

\begin{tikzpicture}
\node at (-4,-0.1) {$(\mathcal{R}_1)$};
\node at (0,0) {$\mathrm{X}+\mathrm{Y} \,\overset{\scriptstyle{k_1}}\longrightarrow\, 2\mathrm{Y}$,};
\node at (3.5,-0.1) {$\mathrm{Y}+\mathrm{Z}\, \overset{\scriptstyle{k_2}}\longrightarrow\, \mathrm{X} \,\, \xrightleftharpoons[\,\,k_3\,\,]{k_4}\,\, 0$};
\node[rotate=45] at (5.8,0.5) {$\xrightleftharpoons[\,\,k_6\,\,]{k_5}\,\, \mathrm{Y}$};
\node[rotate=-45] at (5.8,-0.7) {$\xrightleftharpoons[\,\,k_8\,\,]{k_7}\,\, \mathrm{Z}$};
\end{tikzpicture}

The reactions are labelled with their rate constants. Setting $k_1=0.5$, $k_2=3.0$, $k_3=2.5$, $k_4=0.2$, $k_5=0.6$, $k_6=2.4$, $k_7=1.8$ and $k_8=0.4$, and choosing initial conditions $\mathrm{X}_0=\mathrm{Y}_0=\mathrm{Z}_0=1$, in numerical simulations the system settles, after initial transient behaviour, onto the periodic orbit shown in Figure~\ref{fig1} below. This is assumed to be a genuine periodic orbit, linearly stable relative to its stoichiometric class (which in this case is the entire nonnegative orthant).

\begin{figure}[h]
\begin{minipage}{0.48\textwidth}
\begin{center}
\includegraphics[scale=.3]{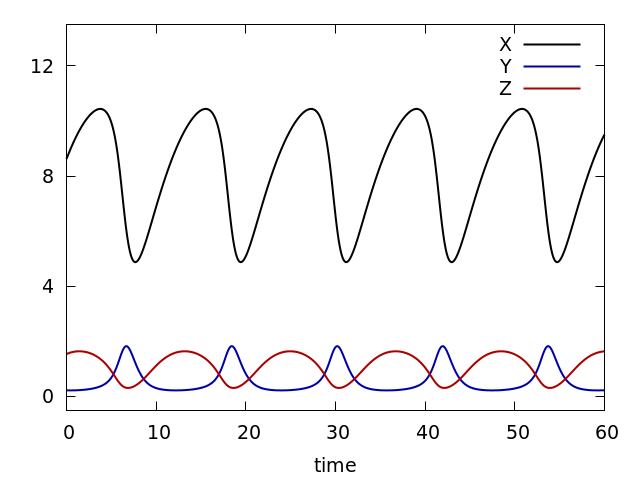}
\end{center}
\end{minipage}
\hfill
\begin{minipage}{0.48\textwidth}
\begin{center}
\includegraphics[scale=.3]{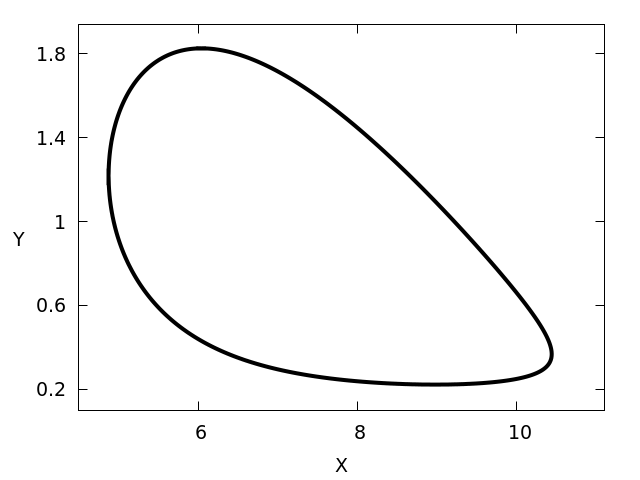}
\end{center}
\end{minipage}
\caption{\label{fig1} Numerical simulation of $\mathcal{R}_1$ with mass action kinetics and rate constants and initial conditions as given in the text. {\em Left.} Evolution of the concentrations of $\mathrm{X}$, $\mathrm{Y}$ and $\mathrm{Z}$. {\em Right.} The projection of the periodic orbit onto $\mathrm{X}\mhyphen \mathrm{Y}$ coordinates.}
\end{figure}

We now split two of the reactions in $\mathcal{R}_1$, namely $\mathrm{X}+\mathrm{Y} \rightarrow 2\mathrm{Y}$ and $\mathrm{X} \rightarrow 0$, and insert intermediate complexes involving three new species $\mathrm{U}, \mathrm{V}$ and $\mathrm{W}$ to obtain the CRN

\begin{tikzpicture}
\node at (-6,-0.1) {$(\mathcal{R}_2)$};
\node at (-2,0) {$\mathrm{X}+\mathrm{Y} \,\overset{\scriptstyle{k_1}}\longrightarrow\, \mathrm{Z} + 2\mathrm{U} \overset{\scriptstyle{k_1'}}\longrightarrow 2\mathrm{Y}$,};
\node at (1.7,0) {$\mathrm{Y}+\mathrm{Z}\, \overset{\scriptstyle{k_2}}\longrightarrow\, \mathrm{X}$};
\draw[->] (2.9,0.1)--(3.8,0.5);
\node at (4,0.8) {$\mathrm{V}+\mathrm{W}$};
\draw[->] (4.2,0.5)--(5.1,0.1);
\draw[->] (5.0,-0.1)--(3.0,-0.1);
\node[rotate=30] at (3.2,0.45) {$\scriptstyle{k_4}$};
\node[rotate=-30] at (4.8,0.45) {$\scriptstyle{k_4'}$};
\node at (4,-0.3) {$\scriptstyle{k_3}$};

\node at (5.35,-0.1) {$0$};
\node[rotate=45] at (6.0,0.5) {$\xrightleftharpoons[\,\,k_6\,\,]{k_5}\,\, \mathrm{Y}$};
\node[rotate=-45] at (6.0,-0.7) {$\xrightleftharpoons[\,\,k_8\,\,]{k_7}\,\, \mathrm{Z}$};
\end{tikzpicture}


Note that $\mathcal{R}_2$ now has a conserved quantity $\mathrm{W}-\mathrm{V}$. The matrix $\beta$ representing the stoichiometries of the new species in the added complexes is 
\[
\left(\begin{array}{rr}2&0\\0&1\\0&1\end{array}\right)\,,
\]
which clearly has rank $2$. Thus the nondegeneracy condition in Theorem~\ref{mainthm} is satisfied. Following the proof in Appendix~\ref{appproof}, we define a positive parameter $\varepsilon$, and set $k_1'=\varepsilon^{-2}$ and $k_4'=\varepsilon^{-1}$ ($2$ and $1$ are the column sums of the top $2 \times 2$ block of $\beta$). The proof tells us that choosing and fixing $\varepsilon>0$ sufficiently small, while keeping rate constants $k_1, \ldots, k_8$ at their previous values, ensures that $\mathcal{R}_2$ has a positive periodic orbit which is linearly stable relative to its stoichiometric class. Some plots of the periodic orbit in the case $\varepsilon=0.1$, and with initial values satisfying $\mathrm{W}_0-\mathrm{V}_0=1$, are shown in Figure~\ref{fig2} below. 

\begin{figure}[h]
\begin{minipage}{0.48\textwidth}
\begin{center}
\includegraphics[scale=.3]{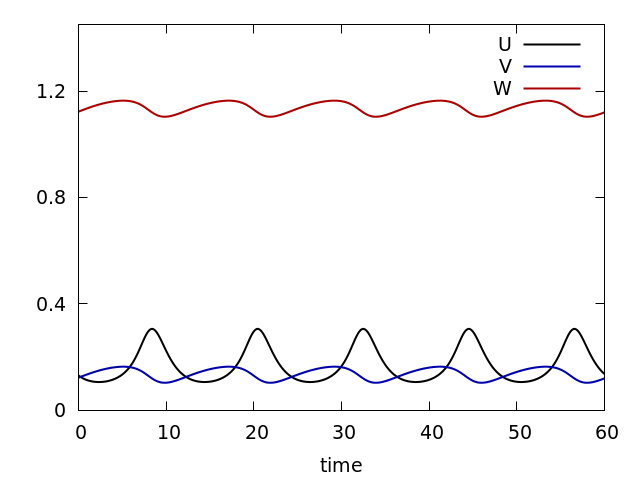}
\end{center}
\end{minipage}
\hfill
\begin{minipage}{0.48\textwidth}
\begin{center}
\includegraphics[scale=.3]{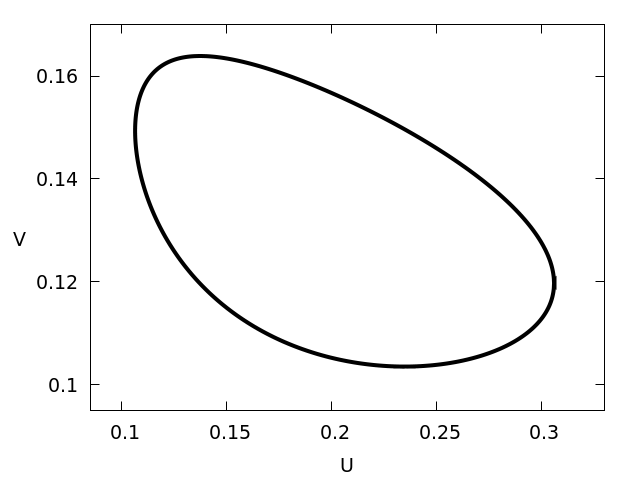}
\end{center}
\end{minipage}
\caption{\label{fig2} Numerical simulation of $\mathcal{R}_2$ with mass action kinetics and rate constants and initial conditions as given in the text. {\em Left.} Evolution of the concentrations of $\mathrm{U}$, $\mathrm{V}$ and $\mathrm{W}$. {\em Right.} The projection of the periodic orbit onto $\mathrm{U}\mhyphen \mathrm{V}$ coordinates.}
\end{figure}

A detailed analysis of this example is given in Appendix~\ref{appproof} after the proof of Theorem~\ref{mainthm}, where it is used to illustrate the proof.

\subsection{The inheritance of multiple nondegenerate limit sets}
\label{secmultiinherit}

Consider the following $3$-species, $4$-reaction, bimolecular CRN:
\begin{equation}
\tag{\mbox{$\mathcal{R}_3$}}
0 \overset{k_1}{\longrightarrow} \mathrm{X}, \quad \mathrm{X}+\mathrm{Y} \overset{k_2}{\longrightarrow} 2\mathrm{Y}, \quad \mathrm{Y} \overset{k_3}{\longrightarrow} 2\mathrm{Z}, \quad \mathrm{X}+\mathrm{Z} \overset{k_4}{\longrightarrow} 0\,.
\end{equation}
In \cite{banajiborosHopf} it was shown, by proving the occurrence of a so-called Bautin bifurcation in the system, that there exist mass action rate constants, say $k_1^*,\, k_2^*,\, k_3^*,\, k_4^*$, at which $\mathcal{R}_3$ simultaneously admits (i) a linearly stable equilibrium; (ii) a linearly stable periodic orbit; and (iii) a nondegenerate, unstable periodic orbit. In fact, $\mathcal{R}_3$ is one of many $3$-species, $4$-reaction, bimolecular, mass action CRNs which are proved, via examination of bifurcations of co-dimension $2$, to allow the coexistence of multiple nondegenerate limit sets. Let us now ``split'' the reaction $\mathrm{X}+\mathrm{Y} \longrightarrow 2\mathrm{Y}$ and introduce a new complex consisting of a single new species $\mathrm{W}$ to obtain the enlarged CRN
\begin{equation}
\tag{\mbox{$\mathcal{R}_4$}}
0 \overset{k'_1}{\longrightarrow} \mathrm{X}, \quad \mathrm{X}+\mathrm{Y} \overset{k'_2}{\longrightarrow} \mathrm{W} \overset{k'_5}{\longrightarrow} 2\mathrm{Y}, \quad \mathrm{Y} \overset{k'_3}{\longrightarrow} 2\mathrm{Z}, \quad \mathrm{X}+\mathrm{Z} \overset{k'_4}{\longrightarrow} 0\,.
\end{equation}
Theorem~\ref{mainthm} then tells us that $\mathcal{R}_4$ admits, simultaneously, a linearly stable equilibrium, a linearly stable periodic orbit, and an unstable periodic orbit. Moreover, we obtain rate constants at which this occurs by fixing $k'_1 = k_1^*,\, k'_2 = k_2^*,\, k'_3 = k_3^*,\, k'_4 = k_4^*$, defining a new, positive parameter $\varepsilon$, setting $k'_5 = \varepsilon^{-1}$, and choosing $\varepsilon$ sufficiently small.

\subsection{A biologically important network: ERK regulation}
\label{secERK}

In \cite{obatake,ConradiERK2}, dynamical behaviours in various models for the phosphorylation and dephosphorylation of extracellular signal-regulated kinase (ERK) are analysed. The authors are interested in how model simplifications preserve the capacity of models of this system for oscillation and multistationarity, and at which parameter values these behaviours must occur in the models. Amongst the networks discussed are the following two:

\begin{enumerate}[align=left,leftmargin=*,itemsep=2ex,parsep=1ex]

%
%
%
%
%
%

\item The reduced ERK network:

\begin{tikzpicture}[scale=0.82]

\node at (-0.2,0) {$\mathrm{S_{00}+E}$};
\node at (1.0,0.3) {$\scriptstyle{k_1}$};
\node at (1.0,0) {$\longrightarrow$};
\node at (2,0) {$\mathrm{S_{00}E}$};
\node at (3.0,0.3) {$\scriptstyle{k_3}$};
\node at (3,0) {$\longrightarrow$};
\node at (4,0) {$\mathrm{S_{01}E}$};
\node at (5.0,0.3) {$\scriptstyle{k_{\mathrm{cat}}}$};
\node at (5,0) {$\longrightarrow$};
\node at (6.2,0) {$\mathrm{S_{11}+ E}$};
\node at (4.5,-0.7) {$\scriptstyle{k_{\mathrm{off}}}$};
\node[rotate=-90] at (4,-0.7) {$\longrightarrow$};
\node at (4,-1.4) {$\mathrm{S_{01}+E}$};
\node at (5.9,-0.7) {$\scriptstyle{m}$};
\node[rotate=90] at (6.2,-0.7) {$\longrightarrow$};
\node at (6.2,-1.4) {$\mathrm{S_{10}+E}$};

\begin{scope}[xshift=9cm]
\node at (-0.2,0) {$\mathrm{S_{11}+F}$};
\node at (1.0,0.3) {$\scriptstyle{\ell_1}$};
\node at (1.0,0) {$\longrightarrow$};
\node at (2,0) {$\mathrm{S_{11}F}$};
\node at (3.0,0.3) {$\scriptstyle{\ell_3}$};
\node at (3,0) {$\longrightarrow$};
\node at (4,0) {$\mathrm{S_{10}F}$};
\node at (5.0,0.3) {$\scriptstyle{\ell_{\mathrm{cat}}}$};
\node at (5,0) {$\longrightarrow$};
\node at (6.2,0) {$\mathrm{S_{00}+ F}$};
\node at (4.5,-0.7) {$\scriptstyle{\ell_{\mathrm{off}}}$};
\node[rotate=-90] at (4,-0.7) {$\longrightarrow$};
\node at (4,-1.4) {$\mathrm{S_{10}+F}$};
\node at (5.9,-0.7) {$\scriptstyle{n}$};
\node[rotate=90] at (6.2,-0.7) {$\longrightarrow$};
\node at (6.2,-1.4) {$\mathrm{S_{01}+F}$};

\end{scope}

\end{tikzpicture}

\item The full ERK network: 

\begin{tikzpicture}[scale=0.82]

\node at (-0.2,0) {$\mathrm{S_{00}+E}$};
\node at (1.0,0.35) {$\scriptstyle{k_1}$};
\node at (1.0,-0.3) {$\bm{\scriptstyle{k_2}}$};
\node at (1.0,0.05) {$\longrightarrow$};
\node[rotate=180] at (1.0,-0.05) {$\bm{\longrightarrow}$};

\node at (2,0) {$\mathrm{S_{00}E}$};
\node at (3.0,0.3) {$\scriptstyle{k_3}$};
\node at (3,0) {$\longrightarrow$};
\node at (4,0) {$\mathrm{S_{01}E}$};
\node at (5.0,0.3) {$\scriptstyle{k_{\mathrm{cat}}}$};
\node at (5,0) {$\longrightarrow$};
\node at (6.2,0) {$\mathrm{S_{11}+ E}$};
\node at (4.5,-0.7) {$\scriptstyle{k_{\mathrm{off}}}$};
\node at (3.6,-0.7) {$\bm{\scriptstyle{k_{\mathrm{on}}}}$};
\node[rotate=90] at (3.95,-0.7) {$\bm{\longrightarrow}$};
\node[rotate=-90] at (4.05,-0.7) {$\longrightarrow$};
\node at (4,-1.4) {$\mathrm{S_{01}+E}$};
\node at (5.85,-0.7) {$\bm{\scriptstyle{m_3}}$};
\node[rotate=90] at (6.2,-0.7) {$\bm{\longrightarrow}$};
\node at (6.2,-1.4) {$\mathrm{S_{10}E}$};
\node at (5.85,-2.1) {$\scriptstyle{m}$};
\node at (6.6,-2.1) {$\bm{\scriptstyle{m_2}}$};
\node[rotate=90] at (6.15,-2.1) {$\longrightarrow$};
\node[rotate=-90] at (6.25,-2.1) {$\bm{\longrightarrow}$};
\node at (6.2,-2.8) {$\mathrm{S_{10}+E}$};

\begin{scope}[xshift=9cm]
\node at (-0.2,0) {$\mathrm{S_{11}+F}$};
\node at (1.0,0.35) {$\scriptstyle{\ell_1}$};
\node at (1.0,-0.3) {$\bm{\scriptstyle{\ell_2}}$};
\node at (1.05,0.05) {$\longrightarrow$};
\node[rotate=180] at (1.05,-0.05) {$\bm{\longrightarrow}$};

\node at (2,0) {$\mathrm{S_{11}F}$};
\node at (3.0,0.3) {$\scriptstyle{\ell_3}$};
\node at (3,0) {$\longrightarrow$};
\node at (4,0) {$\mathrm{S_{10}F}$};
\node at (5.0,0.3) {$\scriptstyle{\ell_{\mathrm{cat}}}$};
\node at (5,0) {$\longrightarrow$};
\node at (6.2,0) {$\mathrm{S_{00}+ F}$};
\node at (4.5,-0.7) {$\scriptstyle{\ell_{\mathrm{off}}}$};
\node at (3.6,-0.7) {$\bm{\scriptstyle{\ell_{\mathrm{on}}}}$};
\node[rotate=90] at (3.95,-0.7) {$\bm{\longrightarrow}$};
\node[rotate=-90] at (4.05,-0.7) {$\longrightarrow$};
\node at (4,-1.4) {$\mathrm{S_{10}+F}$};
\node at (5.85,-0.7) {$\bm{\scriptstyle{n_3}}$};
\node[rotate=90] at (6.2,-0.7) {$\bm{\longrightarrow}$};
\node at (6.2,-1.4) {$\mathrm{S_{01}F}$};
\node at (5.85,-2.1) {$\scriptstyle{n}$};
\node at (6.6,-2.1) {$\bm{\scriptstyle{n_2}}$};
\node[rotate=90] at (6.15,-2.1) {$\longrightarrow$};
\node[rotate=-90] at (6.25,-2.1) {$\bm{\longrightarrow}$};
\node at (6.2,-2.8) {$\mathrm{S_{01}+F}$};

\end{scope}

\end{tikzpicture}

\end{enumerate}

Rate constants have been named similarly to in \cite{ConradiERK2}, though a couple of minor changes have been made to facilitate the discussion to follow. The details of the chemical species and processes involved are in \cite{obatake,ConradiERK2}. The authors observe that both the full and reduced ERK networks are capable of oscillation and, in \cite{ConradiERK2}, demonstrate a range of parameter values at which Andronov--Hopf bifurcations occur in the reduced network.

Here, we merely observe that we can build the full network from the reduced network by adding the intermediates $\mathrm{S_{10}E}$ and $\mathrm{S_{01}F}$ into two reactions (corresponding to enlargement E6), and then adding some linearly dependent reactions (corresponding to enlargment E1). We can thus deduce, via Theorem~\ref{thminherit}, that if the reduced ERK network admits linearly stable oscillation, so must the full ERK network. 

Moreover, it follows from the proofs that in order to obtain a linearly stable periodic orbit in the full network we need only to control some of the rate constants (those shown in bold in the full network above). More precisely, suppose that a linearly stable periodic orbit occurs in the reduced ERK network at some values of the rate constants, say $k_1^*, k_3^*, k_{\mathrm{cat}}^*, k_{\mathrm{off}}^*, m^*, \ell_1^*, \ell_3^*, \ell_{\mathrm{cat}}^*, \ell_{\mathrm{off}}^*, n^*$; then we can set the rate constants $k_1, k_3, k_{\mathrm{cat}}, k_{\mathrm{off}}, m, \ell_1, \ell_3, \ell_{\mathrm{cat}}, \ell_{\mathrm{off}}$ and $n$ in the full ERK network to be equal to these values; and then choose the remaining rate constants, $k_2, k_{\mathrm{on}}, m_2, m_3, \ell_2, \ell_{\mathrm{on}}, \ell_2$ and $\ell_3$, in such a way as to guarantee the occurrence of a linearly stable periodic orbit. These observations answer a question posed in Remark~5.3 in \cite{ConradiERK2} in the affirmative. This does not, of course, imply that parameter regions obtained in this way are the {\em only} regions where we might find stable oscillation in the full network. Indeed, as demonstrated by the examples in \cite{banajiboroshofbauer}, enlargements of a CRN which preserve some dynamical behaviours can often introduce further nontrivial behaviours, including new bifurcations.

\subsection{A biologically important network: the MAPK cascade}
\label{secMAPK}

In \cite{banajipanteaMPNE}, a version of the Huang--Ferrell model of the mitogen-activated protein kinase (MAPK) cascade \cite{Huang.1996aa} with negative feedback was discussed in relation to its capacity for multiple equilibria. The reaction network can be represented as in Figure~\ref{figMAPK} (taken from \cite{banajipanteaMPNE}). 

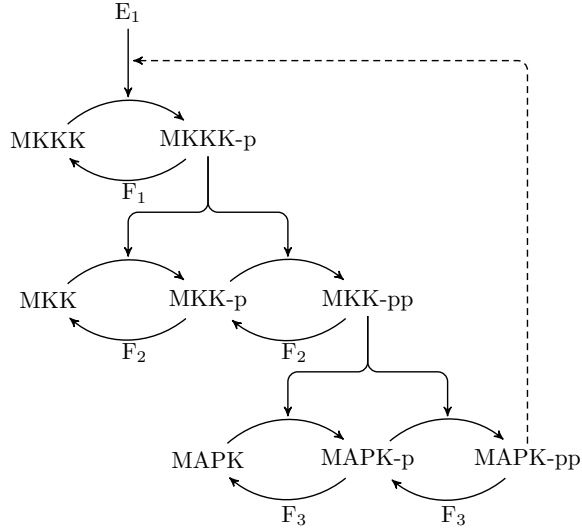
\begin{figure}[h]
\begin{center}
\resizebox{8cm}{!}{
\begin{tikzpicture}
[scale=2.4, 
place/.style={circle,draw=blue!50,fill=blue!20,thick,inner sep=0pt,minimum size=5.5mm},
enzyme/.style={circle,draw= blue!50,fill=white,thick,inner sep=0pt,minimum size=5.5mm},
pre/.style={<-,shorten <=1pt,>=stealth',semithick},
post/.style={->,shorten >=1pt,>=stealth',semithick}]


\node (MAPK) at (-1, 0) {MAPK};
\node (MAPK-pp) at (1,0) {MAPK-pp};
\node (MAPK-p) at (0,0) {MAPK-p}
edge [bend left = 40, post]  (MAPK)
edge [bend right = 40, pre] (MAPK-pp)
edge [bend right = 40, pre]  (MAPK)
edge [bend left = 40, post] (MAPK-pp);

\node (MKK) at (-2, 1) {MKK};
\node (MKK-pp) at (0,1) {MKK-pp};
\node (MKK-p) at (-1,1) {MKK-p}
edge [bend left = 40, post]  (MKK)
edge [bend right = 40, pre] (MKK-pp)
edge [bend right = 40, pre]  (MKK)
edge [bend left = 40, post] (MKK-pp);

\node (MKKK) at (-2, 2) {MKKK};
\node (MKKK-p) at (-1,2) {MKKK-p}
edge [bend left = 40, post]  (MKKK)
edge [bend right = 40, pre]  (MKKK);

\node (Ras/MKKKK) at (-1.5, 2.8) {E$_1$};

\node (F1) at (-1.46, 1.67) {$\text{F}_1$};
\node (F2) at (-1.46, .67) {$\text{F}_2$};
\node (F22) at (-.46, .67) {$\text{F}_2$};
\node (F3) at (-.46, -.35) {$\text{F}_{3}$};
\node (F32) at (.54, -.35) {$\text{F}_{3}$};

\draw [semithick, ->,>=stealth', rounded corners=1.5mm]  (0,.9) -- (0,0.55) -- (0.5,0.55) -- (0.5,0.27);
\draw [semithick, ->,>=stealth', rounded corners=1.5mm]  (0,.9) -- (0,0.55) -- (-0.5,0.55) -- (-0.5,0.27);

\draw [semithick, ->,>=stealth', rounded corners=1.5mm]  (-1,1.9) -- (-1,1.55) -- (-0.5,1.55) -- (-0.5,1.27);
\draw [semithick, ->,>=stealth', rounded corners=1.5mm]  (-1,1.9) -- (-1,1.55) -- (-1.5,1.55) -- (-1.5,1.27);

\draw [semithick, ->,>=stealth', rounded corners=1mm]  (-1.5,2.7) -- (-1.5,2.27);

\draw [densely dashed,semithick, ->,>=stealth', rounded corners=1.5mm] (1,0.1) -- (1,2.5) -- (-1.47,2.5) ;

\end{tikzpicture}
}
\caption{\label{figMAPK}The MAPK cascade with negative feedback \cite{Huang.1996aa} admits linearly stable oscillation with mass action kinetics. Such oscillation can easily be found in simulations for various choices of rate constants. The details of different chemical species involved, and biological motivations for the study of this network, are given in \cite{Kholodenko.2000aa}.}
\end{center}
\end{figure}
This rank-17 network, which we'll term $\mathcal{R}_{\mathrm{MAPK}}$, involves 36 irreversible reactions on 24 chemical species. (The number of species was incorrectly stated as 25 in \cite{banajipanteaMPNE}.) With the abbreviations X = MAPK, Y = MKK, Z = MKKK, it can be written as 9 subnetworks as follows:
\begin{eqnarray*}
&\mbox{(a)}\,\,\,&\text{E$_1$$\,\,+\,\,$Z}\,\,\rightleftharpoons\,\, \text{E$_1$--Z}\,\,\to\,\, \text{E$_1$$\,\,+\,\,$Z-p}, \\
&\mbox{(b)}\,\,\,&\text{F$_1$$\,\,+\,\,$Z-p}\,\,\rightleftharpoons\,\, \text{F$_1$--Z-p}\,\,\to\,\, \text{F$_1$$\,\,+\,\,$Z}\\
&\mbox{(c)}\,\,\,&\text{Z-p$\,\,+\,\,$Y}\,\,\rightleftharpoons\,\, \text{Z-p--Y}\,\,\to\,\, \text{Z-p$\,\,+\,\,$Y-p}\,\,\rightleftharpoons\,\, \text{Z-p--Y-p}\,\,\to\,\, \text{Z-p$\,\,+\,\,$Y-pp}\\  
&\mbox{(d)}\,\,\,& \text{F$_2$$\,\,+\,\,$Y-pp}\,\,\rightleftharpoons\,\, \text{F$_2$--Y-pp}\,\,\to\,\, \text{F$_2$$\,\,+\,\,$Y-p}\,\,\rightleftharpoons\,\, \text{F$_2$--Y-p}\,\,\to\,\, \text{F$_2$$\,\,+\,\,$Y}\\ 
&\mbox{(e)}\,\,\,&\text{Y-pp$\,\,+\,\,$X}\,\,\rightleftharpoons\,\, \text{Y-pp--X} \,\,\to\,\,\text{Y-pp$\,\,+\,\,$X-p}\,\,\rightleftharpoons\,\,\text{Y-pp--X-p}\,\,\to\,\,\text{Y-pp$\,\,+\,\,$X-pp}\\ 
&\mbox{(f)}\,\,\,&\text{F$_3$$\,\,+\,\,$X-pp}\,\,\rightleftharpoons\,\, \text{F$_3$--X-pp}\,\,\to\,\, \text{F$_3$$\,\,+\,\,$X-p}\,\,\rightleftharpoons\,\, \text{F$_3$--X-p}\,\,\to\,\,\text{F$_3$$\,\,+\,\,$X}\\
&\mbox{(g)}\,\,\,&\text{E}_1\,\,+\,\,\text{X-pp}\,\,\rightleftharpoons\,\, \text{E}_1\text{--X-pp}\\  
&\mbox{(h)}\,\,\,&\text{E}_1\text{--X-pp$\,\,+\,\,$Z}\,\,\rightleftharpoons\,\, \text{E}_1\text{--X-pp--Z}\\ 
&\mbox{(i)}\,\,\,&\text{E}_1\text{--X-pp--Z}\,\,\rightleftharpoons\,\, \text{E}_1\text{--Z$\,\,+\,\,$X-pp}.
\end{eqnarray*}
The final three reversible reactions describe the negative feedback process by which the terminal phosphorylated enzyme inhibits the initial phosphorylation step: without these three reactions we have the original network of Huang and Ferrell \cite{Huang.1996aa}. A version of $\mathcal{R}_{\mathrm{MAPK}}$ using Michaelis-Menten kinetics was shown to admit oscillations by Kholodenko in \cite{Kholodenko.2000aa}. Subsequently, oscillations were found by Qiao {\em et al} \cite{Qiao.2007aa} in numerical simulations of the original mass action system of Huang and Ferrell, demonstrating that the negative feedback loop was not necessary for autonomous oscillations to arise in the system. Later, it was shown by Hell and Rendall in \cite{hellrendall2016} that oscillations could be predicted in the Huang--Ferrell model as a consequence of an Andronov--Hopf bifurcation inherited from a smaller network. We remark that a version of the Huang--Ferrell model with {\em positive} feedback was also shown to admit oscillation by Gedeon and Sontag \cite{gedeonsontag2007} using different techniques.

The question we pose here is whether oscillation in $\mathcal{R}_{\mathrm{MAPK}}$ can be inferred, via the inheritance results in Theorem~\ref{thminherit}, from oscillation in any simpler network. We show by example that the answer is yes. Consider the following smaller network with 8 chemical species, 14 irreversible reactions, and rank 8, which we'll term $\mathcal{R}''_{\mathrm{MAPK}}$.
\begin{eqnarray*}
&\mbox{(a)}\,\,\,&\text{E$_1$}\,\,\to\,\, \text{E$_1$$\,\,+\,\,$Z-p},\\
&\mbox{(b)}\,\,\,&\text{Z-p}\,\,\to\,\, 0\\
&\mbox{(c)}\,\,\,&\text{Z-p}\,\,\to\,\, \text{Z-p$\,\,+\,\,$Y-p}\,\,\to\,\, \text{Z-p$\,\,+\,\,$Y-pp}\\  
&\mbox{(d)}\,\,\,& \text{F$_2$$\,\,+\,\,$Y-pp}\,\,\to\,\, \text{F$_2$$\,\,+\,\,$Y-p}\,\,\to\,\, 0\,\,\to\,\, \text{F$_2$}\\ 
&\mbox{(e)}\,\,\,&\text{Y-pp$\,\,+\,\,$X}\,\,\to\,\,\text{Y-pp$\,\,+\,\,$X-p}\,\,\to\,\,\text{Y-pp$\,\,+\,\,$X-pp}\\ 
&\mbox{(f)}\,\,\,&\text{X-pp}\,\,\to\,\, 0\,\,\to\,\, \text{X-p}\,\,\to\,\,\text{X}\\
&\mbox{(g)}\,\,\,&\text{E}_1\,\,+\,\,\text{X-pp}\,\,\rightleftharpoons\,\, 0  
%
%
\end{eqnarray*}
Observe that the negative feedback process has been considerably simplified, but not entirely removed, in $\mathcal{R}''_{\mathrm{MAPK}}$. Although we have not attempted to prove the occurrence of bifurcations leading to oscillation, in numerical simulations we readily find periodic behaviour in the simpler network with mass action kinetics. We now show that if $\mathcal{R}''_{\mathrm{MAPK}}$ indeed admits linearly stable oscillation with mass action kinetics, then $\mathcal{R}_{\mathrm{MAPK}}$ must inherit this behaviour according to Theorem~\ref{thminherit}. 

First, by repeated application of enlargement E3, we add to $\mathcal{R}''_{\mathrm{MAPK}}$ six linearly dependent species, $\mathrm{Z}$, $\mathrm{Y}$, $\mathrm{F}_3$, $\text{F$_2$--Y-p}$, $\text{F$_3$--X-pp}$, and $\text{E}_1\text{--X-pp}$, and arrive at the following network with 14 chemical species, 14 irreversible reactions, and rank 8, which we'll term $\mathcal{R}'_{\mathrm{MAPK}}$:
\begin{eqnarray*}
&\mbox{(a)}\,\,\,&\text{E$_1$$\,\,+\,\,$Z}\,\,\to\,\, \text{E$_1$$\,\,+\,\,$Z-p},\\
&\mbox{(b)}\,\,\,&\text{Z-p}\,\,\to\,\, \text{Z}\\
&\mbox{(c)}\,\,\,&\text{Z-p$\,\,+\,\,$Y}\,\,\to\,\, \text{Z-p$\,\,+\,\,$Y-p}\,\,\to\,\, \text{Z-p$\,\,+\,\,$Y-pp}\\  
&\mbox{(d)}\,\,\,& \text{F$_2$$\,\,+\,\,$Y-pp}\,\,\to\,\, \text{F$_2$$\,\,+\,\,$Y-p}\,\,\to\,\, \text{F$_2$--Y-p}\,\,\to\,\, \text{F$_2$$\,\,+\,\,$Y}\\ 
&\mbox{(e)}\,\,\,&\text{Y-pp$\,\,+\,\,$X}\,\,\to\,\,\text{Y-pp$\,\,+\,\,$X-p}\,\,\to\,\,\text{Y-pp$\,\,+\,\,$X-pp}\\ 
&\mbox{(f)}\,\,\,&\text{F$_3$$\,\,+\,\,$X-pp}\,\,\to\,\, \text{F$_3$--X-pp}\,\,\to\,\, \text{F$_3$$\,\,+\,\,$X-p}\,\,\to\,\,\text{F$_3$$\,\,+\,\,$X}\\
&\mbox{(g)}\,\,\,&\text{E}_1\,\,+\,\,\text{X-pp}\,\,\rightleftharpoons\,\, \text{E}_1\text{--X-pp} 
%
%
\end{eqnarray*}
Next, each of the subnetworks (a) to (g) of $\mathcal{R}_{\mathrm{MAPK}}$ is obtained from the corresponding subnetwork of $\mathcal{R}'_{\mathrm{MAPK}}$ via some sequence of enlargements of the form E1--E6, or using Corollary~\ref{corenz}:
\begin{enumerate}
\item[(a)] We add a new intermediate (E6) and the reverse of an irreversible reaction (E1).
\item[(b)] We introduce an enzymatic mechanism (Corollary~\ref{corenz}).
\item[(c)] We add two new intermediates (E6) and the reverse of two irreversible reactions (E1).
\item[(d)] We add one new intermediate (E6) and the reverse of two irreversible reactions (E1).
\item[(e)] We add two new intermediates (E6) and the reverse of two irreversible reactions (E1).
\item[(f)] We add one new intermediate (E6) and the reverse of two irreversible reactions (E1).
\item[(g)] This reaction is left unchanged.
\end{enumerate}
Finally, the remaining subnetworks, (h) and (i), of $\mathcal{R}_{\mathrm{MAPK}}$ are added in.
\begin{enumerate}
\item[(h)] We add this reversible reaction involving a new complex to the network (E5).
\item[(i)] We add this linearly dependent reaction on existing species (E1).
\end{enumerate}

Note that enlargement E6, the subject of Theorem~\ref{mainthm} here, figures heavily in this example. Indeed, from a practical point of view, adding or removing intermediates appears to be one of the most common and natural operations carried out on reaction networks. 

We do not claim that $\mathcal{R}''_{\mathrm{MAPK}}$ is minimal amongst oscillatory ``subnetworks'' of $\mathcal{R}_{\mathrm{MAPK}}$ w.r.t. to the partial order on CRNs generated by the transformations in Theorem~\ref{thminherit}. This example does, however, demonstrate how stable oscillation in a fairly complex reaction network can be predicted from stable oscillation in a much smaller network, highlighting the utility of Theorem~\ref{thminherit}. An open question is whether oscillation in the original model of Huang and Ferrell \cite{Huang.1996aa}, proved to occur by Hell and Rendall in \cite{hellrendall2016}, can be predicted from oscillation in a smaller network using Theorem~\ref{thminherit}.

\section{Conclusions and extensions}

We have shown in Theorem~\ref{mainthm} that adding intermediate complexes to the reactions of a CRN preserves its capacity for nondegenerate or linearly stable multistationarity and oscillation, provided the new species enter nontrivially into the enlarged network. This completes the task of extending inheritance results in \cite{banajipanteaMPNE} to oscillation. For convenience, we collected together a number of inheritance results in Theorem~\ref{thminherit}. The power of these results, especially when used together, was illustrated in the examples.

The proofs of Theorem~\ref{mainthm} and of several claims summarised in Theorem~\ref{thminherit} rely on GSPT as developed by Fenichel \cite{Fenichel79}. The generality and power of perturbation theory approaches implies various extensions to the inheritance results here. Some of these extensions were remarked on in \cite{banajiboroshofbauer}. For instance, we can go beyond equilibria and periodic orbits and consider the inheritance of other compact, normally hyperbolic invariant manifolds. 

In a related direction, the techniques of perturbation theory permit us to make claims on the inheritance of local bifurcations when we enlarge CRNs. Where the enlargements give rise to regular perturbation problems, the inheritance of bifurcations in enlarged networks can be almost immediate (see remarks and examples in \cite{banajiCRNosci,banajiboroshofbauer}). However where GSPT is required the questions are more subtle, as can be seen in nontrivial applications involving Andronov--Hopf bifurcation and Bogdanov--Takens bifurcation in \cite{hellrendall2016,kreusserrendall2021}. While we have not explicitly treated the inheritance of bifurcations, many of the techniques in this work can be adapted to prove that local bifurcations survive various network enlargements. This would allow us to fully answer questions about the inheritance of bifurcations posed by Conradi and Shiu in \cite{CONRADI2018507}. Making these claims precise remains a task for the future.

Theorem~\ref{thminherit} also suggests extensions to the algorithmic work in \cite{banajiCRNosci}. Such work would involve identifying CRNs which admit nontrivial behaviours such as stable oscillation or multistationarity, and which are minimal with respect to the enlargements E1--E6. The identification of these minimal networks with interesting behaviours should be coupled with algorithms to decide whether a given CRN includes another as a subnetwork, in the sense that the larger CRN can be obtained from the smaller one via a chain of enlargements of the form E1--E6. 

For an example of work in this direction, see \cite{banajiborosHopf}, where bimolecular, mass action networks of minimal size admitting Andronov--Hopf bifurcation are fully enumerated. Using some of the inheritance results gathered in this paper, inferences can immediately be drawn on the occurrence of oscillation in larger networks.

\appendix

\section{Proof of Theorem~\ref{mainthm}}
\label{appproof}

Note that several calculations in the proof are omitted if they are similar to calculations in \cite{banajiCRNosci1}. Notation follows that in \cite{banajipanteaMPNE} and \cite{banajiCRNosci1}. In brief:
\begin{enumerate}
\item $\mathbf{1}$ denotes a vector of ones whose length is inferred from the context. If $\varepsilon$ is any real constant, then $\bm{\varepsilon}$ denotes a vector whose entries are all $\varepsilon$ and whose length is inferred from the context.\item $x^a$ is an abbreviation for the (generalised) monomial $\prod_ix_i^{a_i}$, and $x^A$ means the vector of (generalised) monomials $(x^{A_1}, x^{A_2}, \ldots, x^{A_m})^{\mathrm{t}}$ where $A_i$ is the $i$th row of a matrix $A$. Dimensions are inferred from the context, and all such expressions must, of course, make dimensional sense. 
\item $A \circ B$ is the entrywise product of matrices (or vectors) $A$ and $B$, assumed to have the same dimensions. $A/B$ is the matrix with $ij$th entry $A_{ij}/B_{ij}$, which is defined provided no entry of $B$ is zero.  
\item $d_{\mathrm{H}}(X, Y)$ denotes the Hausdorff distance between two nonempty sets $X$ and $Y$ in $\mathbb{R}^n$ with the Euclidean metric.
\end{enumerate}

\begin{myproof}{Theorem~\ref{mainthm}} Let $\mathcal{R}$ have dynamics governed by the system of ODEs
\begin{equation}
\label{eqbasic0}
\dot x = \Gamma v(x)\,.
\end{equation}
Following Remark~\ref{remkin}, we now observe that the assumption that $\mathcal{R}$ has mass action kinetics plays no part in the proof, beyond ensuring that the rate function $v$ is well defined and $C^2$ on the positive orthant, and is a positive function (i.e., $v\colon \mathbb{R}^n_{+} \to \mathbb{R}^{r}_{+}$). Henceforth this is all we will assume about $v$. However, assuming mass action kinetics for the reactions whose reactant complexes include the new species $\mathrm{Y}$ is important at some points in the proof. 

We assume, without loss of generality, that each reaction of $\mathcal{R}$ is irreversible and, if necessary, trivial reactions have been added, and individual reactions have been written as multiple reactions, as described in Remark~\ref{remcons}.

We also assume, in accordance with the main premise of the theorem, that $v$ is such that (\ref{eqbasic0}) has $r_1$ nondegenerate equilibria and $r_2$ positive nondegenerate periodic orbits, denoted by $\mathcal{O}_{1}, \ldots, \mathcal{O}_{r_1+r_2}$, on some positive stoichiometric class. Let $S$ be the coset of $\mathrm{im}\,\Gamma$ which includes all the $\mathcal{O}_i$. Let $\mathcal{Z}_{\mathcal{O}_i} \subseteq S$ be a relatively open neighbourhood of $\mathcal{O}_i$ in $S$, with the sets $\mathcal{Z}_{\mathcal{O}_i}$ chosen such that their closures are compact, connected, positive, and disjoint. We now choose one limit set, say $\mathcal{O}_j$, and, to reduce notational complexity, denote it by $\mathcal{O}$ and its chosen neighbourhood in $S$ by $\mathcal{Z}_\mathcal{O}$.

Let the $i$th reaction to be split be reaction $s_i$, and define $\underline{v}(x) := (v_{s_1}(x), \ldots, v_{s_m}(x))^\mathrm{t}$ to be the vector of reaction rates associated with the reactions to be split. Define $\alpha_i := c_i - b_i$, $(i = 1, \ldots, m)$ and $\alpha := (\alpha_1|\alpha_2|\cdots |\alpha_m)$. Since, by assumption, $\beta : = (\beta_1|\beta_2|\cdots|\beta_m)$ has rank $m$, we may assume, by reordering the added species $\mathrm{Y}$ if necessary, that the top $m \times m$ block of $\beta$, which we term $\hatt{\beta}$, is nonsingular. We'll refer to the bottom $k \times m$ block of $\beta$ as $\doublehat{\beta}$. In the case that $k=0$, $\doublehat{\beta}$ is an empty matrix. This special case is easier than the general one and requires some remarks and conventions on empty matrices, as described in \cite{banajiCRNosci1}. From here on, we assume that $k>0$, and leave the minor modifications needed to handle the case $k=0$ to the reader.

Assuming mass action kinetics for the reactions $c_i \cdot \mathrm{X} + \beta_i\cdot \mathrm{Y} \rightarrow b_i \cdot \mathrm{X}$, set their rates to be $\ell_ix^{c_i}y^{b_i}$ ($i = 1, \ldots, m$), so that their collective rate is $\ell \circ x^{c^\mathrm{t}} \circ y^{\beta^\mathrm{t}}$ where $\ell := (\ell_1, \ldots, \ell_m)^\mathrm{t}$ and $c:=(c_1|c_2|\cdots|c_m)$. We allow the vector of rate constants $\ell$ to depend on a positive parameter $\varepsilon$, and set $\ell=\bm{\varepsilon}^{-\hatt{\beta}^\mathrm{t}}$. Define
\[
q(x,y,\varepsilon) := \underline{v}(x) - \bm{\varepsilon}^{-\hatt{\beta}^\mathrm{t}} \circ x^{c^\mathrm{t}} \circ y^{\beta^\mathrm{t}}\,.
\]
The assumptions so far give us that the dynamics of $\mathcal{R}'$ is governed by
\begin{equation}
\label{perteq}
\begin{array}{rcl}
\dot x &=& \Gamma v(x) + \alpha q(x,y,\varepsilon)\\
\dot y &=& \beta q(x,y,\varepsilon)\,.
\end{array}
\quad = \quad \Gamma'\left(\begin{array}{c}v(x)\\q(x,y,\varepsilon)\end{array}\right)\,,
\end{equation}
where
\[
\Gamma':=\left(\begin{array}{cc}\Gamma&\alpha\\0&\beta\end{array}\right)\,.
\]
Note that $\Gamma'$ is not actually the stoichiometric matrix of $\mathcal{R}'$ (this was erroneously stated in the proof of Theorem~6 in \cite{banajipanteaMPNE}); but $\Gamma'$ is easily seen to be obtained from the stoichiometric matrix of $\mathcal{R}'$ via elementary column operations, and hence the image of $\Gamma'$ is the stoichiometric subspace of $\mathcal{R}'$, which is all that is needed. Our goal is to show that, for each sufficiently small $\varepsilon>0$, (\ref{perteq}) admits a positive equilibrium or periodic orbit (later denoted by $\mathcal{O}'_{\varepsilon}$) which is nondegenerate relative to the coset of $\mathrm{im}\,\Gamma'$ (later denoted by $S'$) in which it lies.  

Define $\delta := -(\doublehat{\beta}\hat{\beta}^{-1})^{\mathrm{t}}$ and note that $\delta^{\mathrm{t}}\hatt{y}+\doublehat{y}$ is constant along trajectories of (\ref{perteq}). We choose this constant to be $\mathbf{1}$ (any other vector in $\mathbb{R}^k_{+}$ will do) and introduce the new variable $z = x - \alpha \hatt{\beta}\hatt{y}$. More formally, define the hyperplane
\[
\mathcal{H}:=\{(x,y) \in \mathbb{R}^n \times \mathbb{R}^{m+k}\,\colon\, \delta^{\mathrm{t}}\hatt{y}+\doublehat{y} = \mathbf{1}\}
\]
and define the affine bijection $\phi\colon \mathcal{H} \to \mathbb{R}^n \times \mathbb{R}^m$ by
\[
\phi(x, (\hatt{y}, \mathbf{1}-\delta^{\mathrm{t}}\hatt{y})) = (x-\alpha\hat{\beta}^{-1} \hatt{y}, \hatt{y})\,.
\]
We refer to the codomain of $\phi$ as $(z, \hatt{y})$-space. Regarding $(z, \hatt{y})$ as coordinates on $\mathcal{H}$, (\ref{perteq}) reduces to:
\begin{equation}
\label{perteq1}
\begin{array}{rcl}
\dot z &=& \Gamma v(z+\alpha\hat{\beta}^{-1} \hatt{y})\\
\dot {\hatt{y}} &=& \hatt{\beta}q(z+\alpha\hat{\beta}^{-1} \hatt{y},(\hatt{y},\mathbf{1}-\delta^{\mathrm{t}}\hatt{y}),\varepsilon)\,.
\end{array}
\end{equation}
We are interested in (\ref{perteq1}) on a domain where $z+\alpha\hat{\beta}^{-1} \hatt{y}$, $\hatt{y}$ and $\mathbf{1}-\delta^{\mathrm{t}}\hatt{y}$ are all positive. We will later ensure these conditions are met.

Let $x_0$ be any point on $S$, and let $S'$ be the coset of $\mathrm{im}\,\Gamma'$ which includes the point $(x,\hatt{y},\doublehat{y})=(x_0, 0, \mathbf{1})$. Note that $S' \subseteq \mathcal{H}$, and $\left. \phi\right|_{S'}$ is an affine bijection between $S'$ and $S \times \mathbb{R}^m$ (the calculations are in \cite{banajiCRNosci1}). We may thus identify $S'$ with $S \times \mathbb{R}^m$. 

As $q(x,(\hatt{y}, \doublehat{y}),\varepsilon) = \underline{v}(x) - \bm{\varepsilon}^{-\hatt{\beta}^\mathrm{t}}\circ \hatt{y}^{\hatt{\beta}^\mathrm{t}}\circ\doublehat{y}^{\doublehat{\beta}^\mathrm{t}}\circ x^{c^\mathrm{t}}$, we may rewrite (\ref{perteq1}) more explicitly as
\begin{equation}
\label{perteq2}
\begin{array}{rcl}
\dot z &=& \Gamma v(z+\alpha\hat{\beta}^{-1} \hatt{y})\\
\dot {\hatt{y}} &=& \hatt{\beta}\left(\underline{v}(z+\alpha\hat{\beta}^{-1} \hatt{y}) - \bm{\varepsilon}^{-\hatt{\beta}^\mathrm{t}}\circ\hatt{y}^{\hatt{\beta}^\mathrm{t}}\circ(\mathbf{1}-\delta^{\mathrm{t}}\hatt{y})^{\doublehat{\beta}^\mathrm{t}}\circ (z+\alpha\hat{\beta}^{-1} \hatt{y})^{c^\mathrm{t}}\right)\,.
\end{array}
\end{equation}
Next, we define a new variable $w=\hatt{y}/\varepsilon$. Observe that this rescaled variable did not figure in the main proof in \cite{banajiCRNosci1}, but is needed here in order to get a singularly perturbed system amenable to the analysis in \cite{Fenichel79}. Formally, define $\psi_\varepsilon\colon \mathbb{R}^n \times \mathbb{R}^m \to \mathbb{R}^n \times \mathbb{R}^m$ by $\psi_\varepsilon(z,\hatt{y})=(z,\hatt{y}/\varepsilon)$ and note that, for each fixed $\varepsilon>0$, $\left. \psi_\varepsilon\right|_{S \times \mathbb{R}^m}$ is a smooth bijection on $S \times \mathbb{R}^m$. 

In terms of $z, w$, and $\varepsilon$, (\ref{perteq2}) becomes the following singularly perturbed system
\[
(A_\varepsilon)\quad \begin{array}{rcl}
\dot z &=& \Gamma v(z+\varepsilon\alpha\hat{\beta}^{-1} w)\\
\varepsilon\dot w &=& \hatt{\beta}\left(\underline{v}(z+\varepsilon\alpha\hat{\beta}^{-1} w) - w^{\hatt{\beta}^\mathrm{t}}\circ(\mathbf{1}-\varepsilon\delta^{\mathrm{t}}w)^{\doublehat{\beta}^\mathrm{t}}\circ (z+\varepsilon\alpha\hat{\beta}^{-1} w)^{c^\mathrm{t}}\right)\,.
\end{array}
\] 
For each fixed value of $\varepsilon>0$, we are interested in $(A_{\varepsilon})$ on a domain where $z+\varepsilon\alpha\hat{\beta}^{-1} w$, $w$ and $\mathbf{1}-\varepsilon \delta^{\mathrm{t}} w$ are all positive. We will ensure these conditions are met.

Rescaling time in the usual way gives the ``fast time'' version of $(A_{\varepsilon})$:
\[
(B_\varepsilon)\quad \begin{array}{rcl}
\dot z &=& \varepsilon\Gamma v(z+\varepsilon\alpha\hat{\beta}^{-1} w)\\
\dot w &=& \hatt{\beta}\left(\underline{v}(z+\varepsilon\alpha\hat{\beta}^{-1} w) - w^{\hatt{\beta}^\mathrm{t}}\circ(\mathbf{1}-\varepsilon\delta^{\mathrm{t}}w)^{\doublehat{\beta}^\mathrm{t}}\circ (z+\varepsilon\alpha\hat{\beta}^{-1} w)^{c^\mathrm{t}}\right)\,.
\end{array}
\] 
In the limit $\varepsilon \to 0+$, $(A_\varepsilon)$ and $(B_\varepsilon)$ become, respectively,
\[
(A_0)\quad \begin{array}{rcl}
\dot z &=& \Gamma v(z)\\
0 &=& \hatt{\beta}(\underline{v}(z) - w^{\hatt{\beta}^\mathrm{t}}\circ z^{c^\mathrm{t}})
\end{array}\quad \mbox{and} \quad (B_0)\quad \begin{array}{rcl}
\dot z &=& 0\\
\dot w &=& \hatt{\beta}\left(\underline{v}(z) - w^{\hatt{\beta}^\mathrm{t}}\circ z^{c^\mathrm{t}}\right)\,.
\end{array}
\] 
We see that ($B_0$) has a manifold of (positive) equilibria 
\[
\{(z,w) \in \mathbb{R}^n_{+} \times \mathbb{R}^m_{+}\,\colon\,\underline{v}(z) = w^{\hatt{\beta}^\mathrm{t}}\circ z^{c^\mathrm{t}}\} = \{(z,w) \in \mathbb{R}^n_{+} \times \mathbb{R}^m_{+}\,\colon\,w=V(z)\circ z^\gamma\}\,,
\]
where $V(z) = (\underline{v}(z))^{(\hatt{\beta}^{-1})^{\mathrm{t}}}$ and $\gamma = -(c\hatt{\beta}^{-1})^{\mathrm{t}}$. Note that whenever $z$ is positive, the same holds for $\underline{v}(z)$, $z^\gamma$ and $V(z)$. Certainly, $V(z)\circ z^\gamma \in \mathbb{R}^m_{+}$ for all $z \in \overline{\mathcal{Z}_\mathcal{O}}$ (the closure of $\mathcal{Z}_\mathcal{O}$). We define
\[
\mathcal{E}_\mathcal{O}:= \{(z,w) \in \mathcal{Z}_\mathcal{O} \times \mathbb{R}^m\,\colon\, w=V(z)\circ z^\gamma\}\,,
\]
and note that $\mathcal{E}_\mathcal{O}$ is a subset of the positive equilibria of $(B_0)$, and is the image of $\mathcal{Z}_\mathcal{O}$ under the map $h \colon \mathbb{R}^n_{+} \to \mathbb{R}^n_{+} \times \mathbb{R}^m_{+}$, $z\mapsto (z,V(z)\circ z^\gamma)$. Indeed, if we restrict its domain to $\mathcal{Z}_\mathcal{O}$ and codomain to $\mathcal{E}_\mathcal{O}$, $h$ is a $C^2$ diffeomorphism. 

We make the following observations:
\begin{enumerate}
\item[(i)] The eigenvalues of the equilibria of $(B_0)$ comprising $\mathcal{E}_\mathcal{O}$ corresponding to directions in $S \times \mathbb{R}^m$ {\em not} tangential to $\mathcal{E}_\mathcal{O}$, are real and negative. To see this we compute the derivative of $F(z,w):=\hatt{\beta}(\underline{v}(z) - w^{\hatt{\beta}^\mathrm{t}}\circ z^{c^\mathrm{t}})$ w.r.t. $w$ and obtain
\[
D_{w}F(z,w) = -\hatt{\beta}\mathrm{diag}(w^{\hatt{\beta}^\mathrm{t}}\circ z^{c^\mathrm{t}})\hatt{\beta}^\mathrm{t}\mathrm{diag}(\mathbf{1}/w)\,,
\]
giving
\[
\left.D_{w}F(z,w)\right|_{(z,w) \in \mathcal{E}_\mathcal{O}} = -\hatt{\beta}\mathrm{diag}(\underline{v}(z))\hatt{\beta}^\mathrm{t}\mathrm{diag}(\mathbf{1}/(V(z)\circ z^\gamma))\,.
\]
The reader may confirm that any product of square matrices of the form $AD_1A^\mathrm{t}D_2$ where $A$ is nonsingular and $D_1, D_2$ are positive diagonal matrices is similar to a positive definite matrix, and hence $-\hatt{\beta}\mathrm{diag}(\underline{v}(z))\hatt{\beta}^\mathrm{t}\mathrm{diag}(\mathbf{1}/(V(z)\circ z^\gamma))$ has real, negative, eigenvalues for any positive $z$. Further details are in \cite{banajiCRNosci1}.

\item[(ii)] The differential-algebraic system $(A_0)$ defines a local flow on $\mathcal{E}_\mathcal{O}$ which is simply the lifting of the local flow of $\dot z = \Gamma z$ on $\mathcal{Z}_\mathcal{O}$ to $\mathcal{E}_\mathcal{O}$ via $h$. The limit set $\mathcal{O}$ of $\dot z = \Gamma v(z)$ is mapped to the limit set $\mathcal{O}^{*}:=h(\mathcal{O})$ on $\mathcal{E}_\mathcal{O}$ for $(A_0)$.
As $\mathcal{O}$ is nondegenerate relative to $S$, and $h|_{\mathcal{Z}_\mathcal{O}}$ is a $C^2$ diffeomorphism between $\mathcal{Z}_\mathcal{O}$ and $\mathcal{E}_\mathcal{O}$, $\mathcal{O}^{*}$ is nondegenerate relative to $\mathcal{E}_\mathcal{O}$. If $\mathcal{O}$ is hyperbolic (resp., linearly stable) relative to $S$, then the same holds for $\mathcal{O}^{*}$ relative to $\mathcal{E}_\mathcal{O}$. In other words, $\mathcal{O}^{*}$ has these properties as a limit set of the ``reduced vector field'' on $\mathcal{E}_\mathcal{O}$ as discussed in \cite{Fenichel79}.
\end{enumerate}

Define $K_\mathcal{O}=2\sup_{z \in \overline{\mathcal{Z}_\mathcal{O}}}\{|V(z)\circ z^\gamma|\}$. Clearly $\mathcal{E}_\mathcal{O}$, and hence $\mathcal{O}^{*}$, lie in the (open) set 
\[
\mathcal{Z}_\mathcal{O}^+ := \{(z,w) \in \mathcal{Z}_\mathcal{O} \times \mathbb{R}^m \,\colon\, w \in \mathbb{R}^m_{+},\,\, |w| < K_\mathcal{O}\}\,.
\]
Our computations (i) and (ii) have got us to the point where we can apply the relevant results of GSPT. Suppose that $\mathcal{O}$ is a periodic orbit (resp., equilibrium). Then, by Theorem~13.1 in \cite{Fenichel79} (resp., Theorem~12.1 in \cite{Fenichel79}), given any $\zeta>0$, there exists $\varepsilon_0>0$ such that, for each $\varepsilon \in [0,\varepsilon_0)$, ($A_\varepsilon$) has a periodic orbit (resp., equilibrium) $\mathcal{O}^{*}_{\varepsilon}$ on $S \times \mathbb{R}^m$ satisfying $d_H(\mathcal{O}^{*}, \mathcal{O}^{*}_{\varepsilon})<\zeta$. Since $\mathcal{O}^{*} \subseteq \mathcal{Z}_\mathcal{O}^+$, we can choose $\varepsilon_0 >0$ to ensure that $\varepsilon \in [0,\varepsilon_0)$ implies that $\mathcal{O}^{*}_{\varepsilon} \subseteq \mathcal{Z}_\mathcal{O}^+$. Moreover, by Theorem~13.2 in \cite{Fenichel79} (resp., Theorem~12.2 in \cite{Fenichel79}), $\varepsilon_0$ can be chosen to ensure that $\mathcal{O}^{*}_{\varepsilon}$ is nondegenerate relative to $S \times \mathbb{R}^m$ and if $\mathcal{O}$ was hyperbolic (resp., linearly stable) relative to $S$, then $\mathcal{O}^{*}_{\varepsilon}$ is hyperbolic (resp., linearly stable) relative to $S \times \mathbb{R}^m$. 

All that remains is to ensure that $\mathcal{O}^{*}_{\varepsilon}$ is indeed positive in the original coordinates $x$ and $y$. For this we need to impose further conditions on $\varepsilon$. Let $D_\mathcal{O}$ be the minimum distance from $\overline{\mathcal{Z}_\mathcal{O}}$ to the boundary of the nonnegative orthant, namely, $D_\mathcal{O}:= \min_{x \in \overline{\mathcal{Z}_\mathcal{O}}, y \in \partial \mathbb{R}^n_{\geq 0}} d(x,y)$, where $d(\cdot,\cdot)$ is the Euclidean metric, and let $\|\cdot\|$ refer to the matrix norm induced by the Euclidean norm. Define
\[
\varepsilon_1 := \min\left\{\frac{D_\mathcal{O}}{2K_\mathcal{O}\|\alpha\hatt{\beta}^{-1}\|}, \,\, \frac{1}{2K_\mathcal{O}\|\delta^{\mathrm{t}}\|},\,\,\varepsilon_0\right\}\,.
\]
Provided $(z,w) \in \mathcal{Z}_\mathcal{O}^+$ and $\varepsilon\leq \varepsilon_1$, we have (i) $|\varepsilon \alpha\hatt{\beta}^{-1}w| \leq D_\mathcal{O}/2$ and hence $z+\varepsilon \alpha\hatt{\beta}^{-1}w \in \mathbb{R}^n_{+}$, and (ii) $|\varepsilon \delta^\mathrm{t}w| \leq \frac{1}{2}$, and hence $\mathbf{1}-\varepsilon \delta^\mathrm{t}w \in \mathbb{R}^k_{+}$.

Now fix any $\varepsilon \in (0, \varepsilon_1)$. Corresponding to $\mathcal{O}^{*}_{\varepsilon}$ for $(A_{\varepsilon})$, (\ref{perteq2}) has the limit set 
\[
\mathcal{O}^{**}_{\varepsilon} := \psi_\varepsilon^{-1}(\mathcal{O}^{*}_{\varepsilon}) = \{(z,\varepsilon w)\,\colon\,(z,w) \in \mathcal{O}^{*}_{\varepsilon}\}\,.
\]
Recall that $\mathcal{O}^{*}_{\varepsilon} \subseteq \mathcal{Z}_\mathcal{O}^+$, and consequently $\mathcal{O}^{**}_{\varepsilon} \subseteq \mathcal{Z}_\mathcal{O} \times \mathbb{R}^m_{+}$. Since $\mathcal{O}^{*}_{\varepsilon}$ is nondegenerate relative to $S \times \mathbb{R}^m$ and, for each fixed $\varepsilon>0$, $\left. \psi_\varepsilon\right|_{S \times \mathbb{R}^m}$ is a smooth diffeomorphism of $S \times \mathbb{R}^m$ to itself, it holds that $\mathcal{O}^{**}_{\varepsilon}$ is nondegenerate relative to $S \times \mathbb{R}^m$. If $\mathcal{O}$ was hyperbolic (resp., linearly stable) relative to $S$, then the same holds for $\mathcal{O}^{*}_{\varepsilon}$ relative to $S \times \mathbb{R}^m$, and consequently the same holds for $\mathcal{O}^{**}_{\varepsilon}$ relative to $S \times \mathbb{R}^m$.

In turn, (\ref{perteq}) has the limit set
\[
\mathcal{O}'_{\varepsilon} := \phi^{-1}(\mathcal{O}^{**}_{\varepsilon}) =\{(x,(\hatt{y}, \doublehat{y}))\,\colon\,(z,w) \in \mathcal{O}^{*}_{\varepsilon},\,\, \hatt{y} = \varepsilon w,\,\, x=z+\alpha \hatt{\beta}\hatt{y},\,\, \doublehat{y} = \mathbf{1}-\delta^{\mathrm{t}}\hatt{y}\}\,.
\]
As $\mathcal{O}^{**}_{\varepsilon} \subseteq \mathcal{Z}_\mathcal{O} \times \mathbb{R}^m_{+}$, we have $\mathcal{O}'_{\varepsilon} \subseteq \phi^{-1}(\mathcal{Z}_\mathcal{O} \times \mathbb{R}^m_{+})$. Moreover, by construction, whenever $\varepsilon \in (0,\varepsilon_1)$ and $(z,w) \in \mathcal{O}^{*}_{\varepsilon} \subseteq \mathcal{Z}_\mathcal{O}^+$, we have $z+\varepsilon \alpha\hatt{\beta}^{-1}w \in \mathbb{R}^n_{+}$ and $\mathbf{1}-\varepsilon \delta^\mathrm{t}w \in \mathbb{R}^k_{+}$. Thus $\mathcal{O}'_{\varepsilon}$ is a {\em positive} limit set of (\ref{perteq}). Since $\mathcal{O}^{**}_{\varepsilon}$ is nondegenerate relative to $S \times \mathbb{R}^m$ and $\phi^{-1}$ is a smooth diffeomorphism taking $S \times \mathbb{R}^m$ to $S'$, $\mathcal{O}'_{\varepsilon}$ is nondegenerate relative to $S'$. If $\mathcal{O}$ was hyperbolic (resp., linearly stable) relative to $S$, then $\mathcal{O}^{**}_{\varepsilon}$ is hyperbolic (resp., linearly stable) relative to $S \times \mathbb{R}^m$, and consequently $\mathcal{O}'_{\varepsilon}$ is hyperbolic (resp., linearly stable) relative to $S'$.

In the arguments above, we had fixed a limit set $\mathcal{O}$. We can now repeat the arguments for each limit set $\mathcal{O}_i$ ($i=1, \ldots, r_1+r_2$) on $S$. $\varepsilon_1$ as constructed above depends on $\mathcal{O}$, but we can define $\varepsilon_1^*$ to be the minimum of these values of $\varepsilon_1$; for each fixed $\varepsilon \in (0,\varepsilon_1^*)$, (\ref{perteq}) then has $r_1$ positive nondegenerate equilibria and $r_2$ positive nondegenerate periodic orbits on $S'$. Let us fix $\varepsilon \in (0,\varepsilon_1^*)$ and denote these limit sets on $S'$ by $\mathcal{O}'_{1}, \ldots, \mathcal{O}'_{r_1+r_2}$. By construction, if $\mathcal{O}_i$ is hyperbolic (resp., linearly stable) relative to $S$, then the same holds for $\mathcal{O}'_i$ relative to $S'$. Note, finally, that the construction has ensured that the $\mathcal{O}'_i$ are all distinct: the $\mathcal{O}_i$ lie in disjoint sets $\mathcal{Z}_{\mathcal{O}_i}$, and hence $\mathcal{O}'_i$ lie in disjoint sets $\phi^{-1}(\mathcal{Z}_{\mathcal{O}_i} \times \mathbb{R}^m_{+})$. This completes the proof.
\end{myproof}

\subsection{Illustrating key points in the proof of Theorem~\ref{mainthm}} 

We use the example in Section~\ref{secbasicex} to illustrate some aspects of the proof of Theorem~\ref{mainthm}, especially the coordinate transformations needed prior to application of GSPT results. Recall that we begin with the following CRN (written so the reactions appear in order):
\[
\mathrm{X}+\mathrm{Y} \,\overset{\scriptstyle{k_1}}\longrightarrow\, 2\mathrm{Y}, \,\, \mathrm{Y}+\mathrm{Z}\, \overset{\scriptstyle{k_2}}\longrightarrow\, \mathrm{X},\,\,0 \, \overset{\scriptstyle{k_3}}\longrightarrow\, X\, \overset{\scriptstyle{k_4}}\longrightarrow\, 0 \, \overset{\scriptstyle{k_5}}\longrightarrow\, Y\, \overset{\scriptstyle{k_6}}\longrightarrow\, 0\, \overset{\scriptstyle{k_7}}\longrightarrow\, Z\, \overset{\scriptstyle{k_8}}\longrightarrow\, 0\,.
\]
This system gives rise to the differential equations
\begin{equation}
\label{origeq}
\left(\begin{array}{c}\dot x\\\dot y\\\dot z\end{array}\right)  =  \left(\begin{array}{rrrrrrrr}-1&1&\phantom{-}1&-1&\phantom{-}0&0&\phantom{-}0&0\\1&-1&0&0&1&-1&0&0\\0&-1&0&0&0&0&1&-1\end{array}\right)\left(\begin{array}{c}k_1xy\\k_2yz\\k_3\\k_4x\\k_5\\k_6y\\k_7\\k_8z\end{array}\right)\,.\\
\end{equation}
The $3 \times 8$ matrix on the right hand side of (\ref{origeq}) is the stoichiometric matrix of the network, which will be termed $\Gamma$, as in the proof of Theorem~\ref{mainthm}. The vector of reaction rates appearing on the right hand side of (\ref{origeq}) will be denoted by $v(x,y,z)$. Observe that the unique positive stoichiometric class of this system is the entire positive orthant. We assume, based on numerical simulations, that when the parameters $k_i$ are given the values in Section~\ref{secbasicex}, then (\ref{origeq}) indeed has a positive, linearly stable, periodic orbit which we will denote by $\mathcal{O}$. 

We now consider the enlarged system obtained by replacing
\[
\mathrm{X}+\mathrm{Y} \,\overset{\scriptstyle{k_1}}\longrightarrow\, 2\mathrm{Y}\quad \mbox{by}\quad \mathrm{X}+\mathrm{Y} \,\overset{\scriptstyle{k_1}}\longrightarrow\, \mathrm{Z} + 2\mathrm{U} \overset{\scriptstyle{k_1'}}\longrightarrow 2\mathrm{Y}\,,
\]
and
\[
X\, \overset{\scriptstyle{k_4}}\longrightarrow\, 0 \quad \mbox{by}\quad \mathrm{X} \,\overset{\scriptstyle{k_4}}\longrightarrow\, \mathrm{V} + \mathrm{W} \overset{\scriptstyle{k_4'}}\longrightarrow 0\,.
\]
It gives rise to the system of differential equations
\begin{equation}
\label{stoich1}
\left(\begin{array}{c}\dot x\\\dot y\\\dot z\\\dot u\\\dot v\\\dot w\end{array}\right)  =  \left(\begin{array}{rrrrrrrrrr}-1&1&\phantom{-}1&-1&\phantom{-}0&0&\phantom{-}0&0&0&0\\-1&-1&0&0&1&-1&0&0&2&0\\1&-1&0&0&0&0&1&-1&-1&0\\2&0&0&0&0&0&0&0&-2&0\\0&0&0&1&0&0&0&0&0&-1\\0&0&0&1&0&0&0&0&0&-1\end{array}\right)\left(\begin{array}{c}k_1xy\\k_2yz\\k_3\\k_4x\\k_5\\k_6y\\k_7\\k_8z\\k_1'zu^2\\k_4'vw\end{array}\right)\,,
\end{equation}
where the first and fourth reaction vectors are changed, and the last two reactions are those whose rates involve the new species $\mathrm{U}, \mathrm{V}$ and $\mathrm{W}$. Following the proof, and setting $k_1'=\varepsilon^{-2}$ and $k_4'=\varepsilon^{-1}$, (\ref{stoich1}) can be rewritten as
\begin{equation}
\label{stoich2}
\settowidth\mylen{100}
\begin{array}{rcl}
\left(\begin{array}{c}\dot x\\\dot y\\\dot z\\\hline\dot u\\\dot v\\\dot w\end{array}\right)  &=&  \left(\begin{array}{@{}c|c@{}}\begin{array}{rrrrrrrr}-1&1&\phantom{-}1&-1&\phantom{-}0&0&\phantom{-}0&0\\1&-1&0&0&1&-1&0&0\\0&-1&0&0&0&0&1&-1\end{array}&\begin{array}{rr}0&\phantom{-}0\\-2&0\\1&0\end{array}\\
\hline
\begin{array}{rrrrrrrr}\phantom{-}0&\phantom{-}0&\phantom{-}0&\phantom{-}0&\phantom{-}0&\phantom{-}0&\phantom{-}0&\phantom{-}0\\0&0&0&0&0&0&0&0\\0&0&0&0&0&0&0&0\end{array}&\begin{array}{rr}\phantom{-}2&\phantom{-}0\\0&1\\0&1\end{array}
\end{array}\right)
\left(\begin{array}{c}k_1xy\\k_2yz\\k_3\\k_4x\\k_5\\k_6y\\k_7\\k_8z\\\hline k_1xy-\varepsilon^{-2}zu^2\\k_4x-\varepsilon^{-1}vw\end{array}\right).\\
&=& \left(\begin{array}{cc}\Gamma&\alpha\\0&\beta\end{array}\right)\left(\begin{array}{c}v(x,y,z)\\q(x,y,z,u,v,w,\varepsilon)\end{array}\right)\,,
\end{array}
\end{equation}
where
\[
\alpha = \left(\begin{array}{rr}0&0\\-2&0\\1&0\end{array}\right), \quad \beta = \left(\begin{array}{rr}2&0\\0&1\\0&1\end{array}\right) \quad \mbox{and} \quad q = \left(\begin{array}{c}k_1xy-\varepsilon^{-2}zu^2\\k_4x-\varepsilon^{-1}vw\end{array}\right)\,.
\]
Observe that the $6 \times 10$ matrix on the right hand side of (\ref{stoich1}) is the stoichiometric matrix of the enlarged system, whereas the $6 \times 10$ matrix on the right hand side of (\ref{stoich2}), referred to as $\Gamma'$ in the proof of Theorem~\ref{mainthm}, is obtained from the stoichiometric matrix via some elementary column transformations. We can choose $\hat{\beta}$ to be the top $2 \times 2$ submatrix of $\beta$ as this has rank $2$, so that $\doublehat{\beta}$ is then the last row of $\beta$. Following the proof of the theorem, the invariant set $\mathcal{H}$ can be chosen as the hyperplane in $(x,y,z,u,v,w)$ space defined by the equation $w = 1+v$, and so restriction to $\mathcal{H}$ corresponds to fixing $w = 1+v$. The variables needed to bring the system into a desirable form are now 
\[
\left(\begin{array}{c}\mathsf{x}\\\mathsf{y}\\\mathsf{z}\end{array}\right) = \left(\begin{array}{c}x\\y\\z\end{array}\right) - \alpha\hat{\beta}^{-1}\left(\begin{array}{c}u\\v\end{array}\right) = \left(\begin{array}{c}x\\y\\z\end{array}\right) - \left(\begin{array}{rr}0&0\\-1&0\\\frac{1}{2}&0\end{array}\right)\left(\begin{array}{c}u\\v\end{array}\right) = \left(\begin{array}{c} x\\y+u\\z - \frac{1}{2} u\end{array}\right)\,,
\]
and
\[
\left(\begin{array}{c}\mathsf{u}\\\mathsf{v}\end{array}\right) = \left(\begin{array}{c}u/\varepsilon\\v/\varepsilon\end{array}\right)\,.
\]
Note that $(\mathsf{x}, \mathsf{y}, \mathsf{z})$ are collectively referred to as ``$z$'' in the proof, while $(\mathsf{u}, \mathsf{v})$ are collectively referred to as ``$w$'' in the proof. In terms of the new variables, (\ref{stoich2}) naturally appears as a singularly perturbed system:
\begin{equation}
\label{stoich5}
\begin{array}{rcl}
\left(\begin{array}{c}\dot{\mathsf{x}}\\\dot{\mathsf{y}}\\\dot{\mathsf{z}}\end{array}\right) &=& \left(\begin{array}{rrrrrrrr}-1&1&\phantom{-}1&-1&\phantom{-}0&0&\phantom{-}0&0\\1&-1&0&0&1&-1&0&0\\0&-1&0&0&0&0&1&-1\end{array}\right)\left(\begin{array}{c}k_1\mathsf{x}(\mathsf{y}-\varepsilon\mathsf{u})\\k_2(\mathsf{y}-\varepsilon\mathsf{u})(\mathsf{z} + \frac{1}{2} \varepsilon\mathsf{u})\\k_3\\k_4\mathsf{x}\\k_5\\k_6(\mathsf{y}-\varepsilon\mathsf{u})\\k_7\\k_8(\mathsf{z} + \frac{1}{2} \varepsilon\mathsf{u})\end{array}\right)\,,\\
\varepsilon\left(\begin{array}{c}\dot{\mathsf{u}}\\\dot{\mathsf{v}}\end{array}\right) &=& \left(\begin{array}{rr}2&0\\0&1\end{array}\right) \left(\begin{array}{c}k_1\mathsf{x}(\mathsf{y}-\varepsilon\mathsf{u})-(\mathsf{z} + \frac{1}{2} \varepsilon\mathsf{u})\mathsf{u}^2\\k_4\mathsf{x}-\mathsf{v}(1+\varepsilon\mathsf{v})\end{array}\right)\,.
\end{array}
\end{equation}
This is the concrete form, in this example, of the system referred to as $(A_{\varepsilon})$ in the proof above. Hence, $(A_0)$ takes the form 
\begin{equation}
\label{stoich6}
\begin{array}{rcl}
\left(\begin{array}{c}\dot{\mathsf{x}}\\\dot{\mathsf{y}}\\\dot{\mathsf{z}}\end{array}\right) &=& \left(\begin{array}{rrrrrrrr}-1&1&\phantom{-}1&-1&\phantom{-}0&0&\phantom{-}0&0\\1&-1&0&0&1&-1&0&0\\0&-1&0&0&0&0&1&-1\end{array}\right)\left(\begin{array}{c}k_1\mathsf{x}\mathsf{y}\\k_2\mathsf{y}\mathsf{z}\\k_3\\k_4\mathsf{x}\\k_5\\k_6\mathsf{y}\\k_7\\k_8\mathsf{z}\end{array}\right)\,,\\
0 &=& \left(\begin{array}{rr}2&0\\0&1\end{array}\right) \left(\begin{array}{c}k_1\mathsf{x}\mathsf{y}-\mathsf{z}\mathsf{u}^2\\k_4\mathsf{x}-\mathsf{v}\end{array}\right)\,.
\end{array}
\end{equation}
Note that in \eqref{stoich6}, $(\mathsf{x}, \mathsf{y}, \mathsf{z})$ evolve according to the original differential equation (\ref{origeq}), and hence, if we fix the rate constants $k_1, \ldots, k_8$ as before, this subsystem has a periodic orbit, presumed to be linearly stable. Rescaling time, the limiting fast time system, termed $(B_0)$ in the proof, takes the form 
\begin{equation}
\label{stoich7}
\begin{array}{rcl}
\left(\begin{array}{c}\dot{\mathsf{x}}\\\dot{\mathsf{y}}\\\dot{\mathsf{z}}\end{array}\right)  &=& 0\,,\\
\left(\begin{array}{c}\dot{\mathsf{u}}\\\dot{\mathsf{v}}\end{array}\right) &=& \left(\begin{array}{rr}2&0\\0&1\end{array}\right) \left(\begin{array}{c}k_1\mathsf{x}\mathsf{y}-\mathsf{z}\mathsf{u}^2\\k_4\mathsf{x}-\mathsf{v}\end{array}\right)\,.
\end{array}
\end{equation}
The manifold of positive equilibria of (\ref{stoich7}) is defined by 
\begin{equation}
\label{poseqmanif}
\{(\mathsf{x}, \mathsf{y}, \mathsf{z}, \mathsf{u}, \mathsf{v}) \in \mathbb{R}^5_{+}\,\colon\,\, \mathsf{u} = \sqrt{k_1 \mathsf{x} \mathsf{y}/\mathsf{z}},\,\, \mathsf{v} = k_4 \mathsf{x}\}.
\end{equation}
This is a three dimensional submanifold of $\mathbb{R}^5_{+}$. We can easily compute the nontrivial eigenvalues of the positive equilibria of (\ref{stoich7}) and find that they are, as expected, real and negative. Via the theory developed by Fenichel in \cite{Fenichel79}, for sufficiently small $\varepsilon > 0$, (\ref{stoich5}) has an invariant ``slow'' manifold, close on compact sets to a portion of the positive equilibrium manifold (\ref{poseqmanif}). Further, for sufficiently small $\varepsilon>0$, (\ref{stoich5}) has a periodic orbit ($\mathcal{O}^*_\varepsilon$ in the proof) on this slow manifold. This periodic orbit is linearly stable: this holds relative to the slow manifold, since it holds for (\ref{origeq}); and it holds relative to directions in $\mathcal{H}$ transverse to the slow manifold as a consequence of the fact that the nontrivial eigenvalues of the positive equilibria of (\ref{stoich7}) are real and negative. 

The remainder of the proof now simply tells us that (again provided $\varepsilon$ is sufficiently small) the system (\ref{stoich2}) in original coordinates $(x,y,z,u,v,w)$ has a corresponding positive periodic orbit $\mathcal{O}'_{\varepsilon}$, which is linearly stable relative to its stoichiometric class (in this case, the set defined by $w-v=1$).

\bibliographystyle{unsrt}

\end{document}